\documentclass[12pt,reqno]{amsart}
\usepackage{amsthm,amsmath,mathtools,amsfonts,mathrsfs,latexsym,ascmac,amssymb,amscd,stmaryrd,comment,bbm}
\usepackage{enumitem}
\usepackage{graphicx,color}
\usepackage[all,cmtip]{xy}
\usepackage{tikz}
\usetikzlibrary{decorations.markings}
\usetikzlibrary{intersections,calc,arrows.meta}
\usepackage[hidelinks]{hyperref}

\newcommand{\R}{{\mathbb{R}}}
\newcommand{\Z}{{\mathbb{Z}}}
\newcommand{\C}{{\mathbb{C}}}

\newcommand{\F}{{\mathbb{F}}}
\newcommand{\Fk}{{\F_k}}
\newcommand{\cE}{{\mathcal{E}}}
\newcommand{\bL}{{\mathbb{L}}}
\newcommand{\fm}{{\mathfrak{m}}}
\newcommand{\cM}{{\mathcal{M}}}
\newcommand{\cO}{{\mathcal{O}}}
\newcommand{\cR}{{\mathcal{R}}}
\newcommand{\cU}{{\mathcal{U}}}
\newcommand{\cW}{{\mathcal{W}}}

\newcommand{\cp}{{\mathbb{C}P^2}}
\newcommand{\del}{{\partial}}
\newcommand{\delb}{{\bar{\partial}}}
\newcommand{\ii}{{\mathbf{i}}}
\newcommand{\rt}{{\sqrt{-1}}}

\newcommand{\grad}{{\mathrm{grad}}}
\newcommand{\Int}{{\mathrm{Int}}}
\newcommand{\Hom}{{\mathrm{Hom}}}

\newcommand{\id}{{\mathrm{id}}}

\newcommand{\Mo}{{\mathit{Mo}}}
\newcommand{\DG}{{\mathit{DG}}}
\newcommand{\Coh}{{\mathit{Coh}}}

\newcommand{\trop}{{\mathrm{trop}}}
\newcommand{\mult}{{\mathrm{mult}}}
\newcommand{\jag}{{\mathrm{jag}}}
\newcommand{\vect}{{\mathrm{vect}}}
\newcommand{\cMj}{{\cM_{\jag}}}

 \newtheorem{theorem}{Theorem}[section]
 \newtheorem{lemma}[theorem]{Lemma}
 \newtheorem{corollary}[theorem]{Corollary}

\theoremstyle{definition}
 \newtheorem{definition}[theorem]{Definition}
 \newtheorem{remark}[theorem]{Remark}

\begin{document}
\title{Multi-valued Morse homotopy for the SYZ mirror of the complex projective plane}

\author{Hayato Nakanishi}
\address{Graduate School of Science, Chiba University, 1-33 Yayoicho, Inage, Chiba, 263-8522 Japan.}
\email{hnakanishi@g.math.s.chiba-u.ac.jp}

\author{Yat-Hin Suen}
\address{Department of Mathematics, National Cheng Kung University, Tainan, 701 Taiwan.}
\email{yhsuen@gs.ncku.edu.tw}
\date{}

\maketitle
\begin{abstract}
In this paper, we propose a definition of the category $\Mo^\mult(P)$ of multi-valued Morse homotopy on $P$ consisting of multi-valued functions associated to Lagrangian multi-sections.
We then show that a full subcategory $\Mo_\cE^\mult(P)$ of $\Mo^\mult(P)$ is $A_\infty$-equivalent to a full subcategory $\DG_\cE^\vect(\cp)$ of the category $\DG^\vect(\cp)$ consisting of holomorphic vector bundles over the complex projective plane $\cp$. As an application, we study the mirror description for global sections of the holomorphic tangent bundle over $\cp$.
\end{abstract} 
\tableofcontents

\markright{Multi-valued Morse homotopy for the SYZ mirror of $\cp$}

\section{Introduction}
Morse homotopy was first introduced by Fukaya \cite{Fuk93} as a Morse theoretical model of the Fukaya category. In \cite{FO97}, Fukaya-Oh proved that the category of Morse homotopy on a closed manifold $B$ consisting of functions is $A_\infty$-equivalent to the Fukaya category of the cotangent bundle $T^*\!B$ consisting of Lagrangian sections.\par
Kontsevich proposed a categorical formulation of mirror symmetry in \cite{Kon95}, which is called the homological mirror symmetry (HMS) conjecture. On the other hand, Strominger-Yau-Zaslow proposed a geometric formulation of mirror symmetry in \cite{SYZ}, which is called the SYZ conjecture. Kontsevich-Soibelman proposed a framework to prove the HMS conjecture based on the SYZ picture and Fukaya-Oh's results for Morse homotopy. Along these lines, Futaki-Kajiura introduced the category $\Mo(P)$ of Morse homotopy on a manifold $P$ with boundary. The category $\Mo(P)$ is one of the generalizations of the weighted Fukaya-Oh category proposed by Kontsevich-Soibelman in \cite{KS01}. Futaki-Kajiura and the first author proved homological mirror symmetry for some toric manifolds by using $Mo(P)$ in \cite{FK21,FK22,Nak24a,Nak24b}. Nishida generalized this framework to toric orbifolds and proved homological mirror symmetry for weighted projective spaces in \cite{Nis25}.\par
In terms of the SYZ conjecture and the HMS conjecture, the mirrors of holomorphic vector bundles are considered to be Lagrangian multi-sections. In a recent work by Oh and the second author \cite{OS24}, they constructed Lagrangian multi-sections that are mirror to toric vector bundles over a toric surface. In particular, the second author discussed a relationship between toric vector bundles, Lagrangian multi-sections, and spectral networks in \cite{Sue24}.\par
In this paper, we propose a definition of the category $\Mo^\mult(P)$ of multi-valued Morse homotopy consisting of multi-valued functions associated to Lagrangian multi-sections, which is a generalization of $\Mo(P)$ proposed by Futaki-Kajiura in \cite{FK21}. In order to treat multi-valued functions, we introduce the notion of a \textit{jagged gradient tree} that is a gradient tree modified by interactions from spectral networks subordinate to Lagrangian multi-sections proposed by the second author in \cite{Sue24}. Such modifications to gradient trajectories are inspired by the notion of a \textit{tropical Morse tree} proposed by Gross-Siebert in \cite{GS12}. We then compare $\Mo^\mult(P)$ with a dg category $\DG^\vect(\cp)$ consisting of holomorphic toric vector bundles over $\cp$.

\subsection{Main results}
Let $P$ be the moment polytope of the complex projective space $\cp$.
We construct the category $\Mo^\mult(P)$ of multi-valued Morse homotopy on $P$ as one of the generalizations of the category $\Mo(P)$ of weighted Morse homotopy proposed in \cite{FK21} to Lagrangian multi-sections. The objects of $\Mo^\mult(P)$ are multi-valued functions associated to Lagrangian multi-sections satisfying some asymptotic condition. The space of morphisms is the $\Z$-graded vector space spanned by the critical points of the multi-valued functions associated to the Lagrangian multi-sections. An $A_\infty$-structure is defined as the modified $A_\infty$-structure of $\Mo(P)$ by interactions of spectral networks subordinated to the Lagrangian multi-sections.\par
On the other hand, the derived category $D^b(\Coh(\cp))$ has a full strongly exceptional collection given by
\begin{equation*}
\cE :=\left(\cO(1), T_{\cp}, \cO(2) \right)
\end{equation*}
where $T_{\cp}$ is the holomorphic tangent bundle. Let $\DG^\vect(\cp)$ be a dg category consisting of holomorphic vector bundles over $\cp$ and $\DG^\vect_\cE(\cp) \subset \DG^\vect(\cp)$ be the full subcategory consisting of $\cE$. Then, we have
\begin{equation*}
D^b(\Coh(\cp)) \simeq Tr(\DG^\vect_\cE(\cp))
\end{equation*}
where $Tr$ is the twisted complex construction by Bondal-Kapranov-Kontsevich \cite{BK,Kon95}.\par
For toric surfaces, Oh and the second author constructed the correspondence between toric vector bundles and mirror Lagrangian multi-sections in \cite{OS24}. Under this correspondence, we denote by the same symbol $\cE$ the collection of the Lagrangian multi-sections mirror to the full strongly exceptional collection $\cE$ of $D^b(\Coh(\cp))$. Also, we denote by $\Mo_\cE^\mult(P)\subset \Mo^\mult(P)$ the full subcategory consisting of $\cE$. Then, the first main result is the following.

\begin{theorem}[Theorem \ref{main theorem prov}]\label{main theorem}
We have an $A_\infty$-equivalence
\begin{equation*}
\Mo^{\mult}_\cE(P)) \simeq \DG^\vect_\cE(\cp).
\end{equation*}
\end{theorem}

 For two $A_\infty$-equivalent $A_\infty$-categories, the induced triangulated categories are equivalent as triangulated categories, thus we obtain the following.

\begin{corollary}[Corollary \ref{main cor prov}]
We have an equivalence of triangulated categories
\begin{equation*}
Tr(\Mo^{\mult}_\cE(P)) \simeq D^b(\Coh(\cp)).
\end{equation*}
\end{corollary}

The second main result is a description of the global sections of the holomorphic tangent bundle $T_\cp$ on the mirror side. Di Rocco-Jabbusch-Smith introduced the notion of a parliament of polytopes for toric bundles in \cite{DJS14}, which provides the description of the global sections of the toric vector bundle as the lattice points in the polytopes. In our formulation of homological mirror symmetry, the space of global sections of $T_\cp$ should be isomorphic to the space $H(\Mo^\mult(P))(L(0),\bL_{T_\cp})$ where $L(0)$ is the zero section and $\bL_{T_\cp}$ is the mirror Lagrangian multi-section of $T_\cp$. By computing the space $\Mo^\mult(P)(L(0),\bL_{T_\cp})$ explicitly, we obtain the correspondence between generators and the following isomorphisms.

\begin{theorem}[Corollary \ref{cohomology of tangent}, Theorem \ref{cohomology of cotangent}]
We have isomorphisms of the cohomologies
\begin{equation*}
\begin{split}
H^k(\Mo^{\mult}(P))(L(0),\bL_{T_{\cp}}) &\cong H^k(\DG^\vect(\cp))(\cO,T_{\cp}).\\
H^k(\Mo^{\mult}(P))(L(0),\bL_{\Omega_{\cp}}) &\cong H^k(\DG^\vect(\cp))(\cO,\Omega_{\cp}).
\end{split}
\end{equation*}
\end{theorem}

 For any toric bundle $E$, we expect that the description of the global sections by using the parliament of polytopes coincides with the generators of the space of the morphism from the zero section $L(0)$ to the mirror Lagrangian multi-section $\bL_E$ in $Mo^\mult(P)$.

\subsection{Organization}
In section \ref{Preliminaries}, we recall the SYZ set-up and the notion of (tropical) Lagrangian multi-sections. In section \ref{complex side}, we explain the complex side of the homological mirror symmetry. Namely, we consider the dg category consisting of holomorphic vector bundles over $\cp$. In section \ref{symplectic side}, we introduce the category of Multi-valued Morse homotopy and prove the main theorem. In section \ref{global sections and its mirror}, we consider the global sections of the holomorphic tangent bundle $T_\cp$ and its mirror. In section \ref{Example of nontrivial jagged gradient tree in CP2}, we give an example of a jagged gradient tree introduced in subsection \ref{mult-morse}. In appendix \ref{local model}, we review the local model of the Lagrangian multi-section and the gradient vector field around the branch point. In appendix \ref{Non-trivial jagged gradient tree in the local model}, we give an example of a jagged gradient tree in the local model.
\\\ \\
\noindent
{\bf Acknowledgements.}
H. Nakanishi would like to thank Hiroshige Kajiura for many valuable comments.
The work of H. Nakanishi was supported by JSPS KAKENHI Grant Number 24KJ0526. The work of Y.-H. Suen was supported by the National Science and Technology Council (Project No. 113-2115-M-006-016-MY2) and partially supported by the Yushan Fellow Program by the Ministry of Education (MOE), Taiwan. (MOE-114-YSFMS-0005-001-P1).

\section{Preliminaries}\label{Preliminaries}
This section contains a brief review of the SYZ construction and Lagrangian multi-sections. Throughout the whole paper, let $N \cong \Z^2$, $N_\R:=N\otimes_\Z\R$, $M:=\Hom_\Z(N, \Z)$, and $M_\R:=M\otimes_\Z\R$.

\subsection{SYZ construction}\label{SYZ construction}
In this subsection, we review the SYZ construction for the complex projective plane $\cp$. For more details see \cite{LYZ00,Cha09,FK21}.\par 
The complex projective plane $\cp$ is defined by
\begin{equation*}
\cp :=\left\{ [z_0:z_1:z_2]\right\} .
\end{equation*}
We take the following open covering $\{U_i\}_i$: 
\begin{equation*}
U_i  =  \left\{[z_0:z_1:z_2] \in \cp\ \middle|\ z_i \neq 0\right\}.
\end{equation*}
The local coordinates in $U_0$ are $(u,v) := (\frac{z_1}{z_0},\frac{z_2}{z_0})$.
The Fubini-Study form $\omega_{FS}$ is expressed in $U_0$ as
\begin{equation*}
\omega_{FS} = \dfrac{1}{2}\delb\del\log(1+u\bar{u}+v\bar{v}).
\end{equation*}
The moment map $\check{\mu}:\cp\to M_\R$ is given by
\begin{equation*}
\check{\mu}([z_0:z_1:z_2]) = \left(\frac{|z_1|^2}{|z_0|^2+|z_1|^2+|z_2|^2}, \frac{|z_2|^2}{|z_0|^2+|z_1|^2+|z_2|^2}\right).
\end{equation*}
The image $\check{\mu}(\cp)$ is called the moment polytope associated to $\omega_{FS}$ (see Figure \ref{poly1}). Namely, the moment polytope $P=\check{\mu}(\cp)$ is given by
\begin{equation*}
P:=\left\{(x^1,x^2)\in M_\R\ \middle|\ 0\leq x^1 \leq 1,\ 0\leq x^2 \leq 1,\ 0\leq x^1 +x^2 \leq 1\right\}.
\end{equation*}
\begin{figure}[h]
\center
\begin{tikzpicture}[scale=0.8]
\draw(0,0)--  (4,0)--  (0,4) --  cycle;
\draw(0,0) node[below left]{$\check{\sigma}_0$};
\draw(4,0) node[below right]{$\check{\sigma}_1$};
\draw(0,4) node[above left]{$\check{\sigma}_2$};
\draw(2,0) node[below]{$\check{\rho}_2$};
\draw(2,2) node[above right]{$\check{\rho}_0$};
\draw(0,2) node[left]{$\check{\rho}_1$};
\end{tikzpicture}\ \ \ \ \ \ 
\begin{tikzpicture}[scale=0.8]
\draw(0,0) --  (-2,0) node[left]{$\rho_1$};
\draw(0,0) --  (0,-2) node[below]{$\rho_2$};
\draw(0,0) --  (1.4,1.4) node[above right]{$\rho_0$};
\draw(-1,-1) node{$\sigma_{0}$};
\draw(1,-1) node{$\sigma_{1}$};
\draw(-1,1) node{$\sigma_{2}$};
\end{tikzpicture}
\caption{The moment polytope $P\subset M_\R$ and the fan over $N_\R$.}
\label{poly1}
\end{figure}
Now, we set $B:=\mathrm{Int}(P)$ and $\cU:=\check{\mu}^{-1}(B)$. Then, $\cU$ is identified with $TN_\R/N$ and equipped with an affine structure by $u=e^{\xi_1+\ii y_1}$ and $v=e^{\xi_2+\ii y_2}$, where $y_1$ and $y_2$ are the fiber coordinates of $\cU$. Hence, the SYZ mirror of the canonical projection $\mathrm{pr}_\xi:\cU\to N_\R$ is given by $Y:=T^*\!N_\R/M$ and $p_{N_\R}:T^*\!N_\R/M \to N_\R$. We denote by $\overline{Y}$ the covering space of $Y$ and the same symbol $p_{N_\R}:\overline{Y}\to N_\R$.
The Fubini-Study form $\omega_{FS}$ is expressed as
\begin{equation*}
\omega_{FS}  = 2\frac{e^{2\xi_1}(1+e^{2\xi_2})d\xi_1\wedge dy_1 -e^{2\xi_1}e^{2\xi_2}d\xi_1\wedge dy_2 - e^{2\xi_1}e^{2\xi_2}d\xi_2\wedge dy_1 +e^{2\xi_2}(1+e^{2\xi_1})d\xi_2\wedge dy_2}{(1+e^{2\xi_1}+e^{2\xi_2})^2}
\end{equation*}
on $\cU$. 
By this expression, the induced metric $\{g^{ij}\}$ on $N_\R$ is given by
\begin{equation*}
\begin{pmatrix}
g^{11} & g^{12}\\
g^{21} & g^{22}
\end{pmatrix}
=
\begin{pmatrix}
\dfrac{2e^{2\xi_1}(1+e^{2\xi_2})}{(1+e^{2\xi_1}+e^{2\xi_2})^2} & \dfrac{- 2e^{2\xi_1}e^{2\xi_2}}{(1+e^{2\xi_1}+e^{2\xi_2})^2} \\
\dfrac{-2e^{2\xi_1}e^{2\xi_2}}{(1+e^{2\xi_1}+e^{2\xi_2})^2} & \dfrac{2e^{2\xi_2}(1+e^{2\xi_1})}{(1+e^{2\xi_1}+e^{2\xi_2})^2}
\end{pmatrix}.
\end{equation*}
Now, we put $\psi:=\frac{1}{2}\log(1+e^{2 \xi_1}+e^{2 \xi_2})$ and then $\frac{\del^2 \psi}{\del \xi_i\del \xi_j}=g^{ij}$. The Legendre transformation $d\psi:N_\R\to B$ is given by
\begin{equation*}
d\psi(\xi_1,\xi_2)
=
\left(\frac{\del\psi}{\del \xi_1} , \frac{\del\psi}{\del \xi_2}\right)
=
\left(\frac{e^{2 \xi_1}}{1+e^{2 \xi_1}+e^{2 \xi_2}}, \frac{e^{2 \xi_2}}{1+e^{2 \xi_1}+e^{2 \xi_2}}\right) =:(x^1,x^2).
\end{equation*}
We set $\pi:=d\psi\circ p_{N_\R}$ and call the coordinates $(x^1,x^2)$ the dual coordinates of $(\xi_1,\xi_2)$ since the Legendre transformation is a diffeomorphism. 
To summarize, we obtain the following diagram:
\begin{equation*}
\xymatrix{
N_\R\times M_\R=T^*\!N_\mathbb{R} \ar@{->}[d]  \ar@{=}[r]  & \overline{Y} \ar@{->}[d]  &  &\\
N_\R\times (M_\R/M)=T^*\!N_\mathbb{R}/M  \ar@{=}[r]  & Y \ar[rd]^{p_{N_\mathbb{R}}} &  &\mathcal{U} \ar[ld]_{\mathrm{pr}_\xi} \ar[rd]^{\check{\mu}|_\cU} & \\
& & N_\mathbb{R}  \ar[rr]^{\cong}_{d\psi} & & B
}
\end{equation*}
The mirror manifold $Y$ is equipped with the standard symplectic structure $\omega=d\xi_1\wedge dy^1+d\xi_2\wedge dy^2$ where $(y^1,y^2)$ is the dual coordinate of $(y_1,y_2)$. By using the dual coordinate $(x^1,x^2)$, we rewrite $\omega$ as
\begin{equation*}
\omega=d\xi_1\wedge dy^1+d\xi_2\wedge dy^2 = \sum_{i,j=1}^2g_{ij}dx^i\wedge dy^j
\end{equation*}
where $\{g_{ij}\}$ is the metric on $B$ which is the inverse matrix of $\{g^{ij}\}$.\par
In this paper, we treat the base space of the SYZ set-up as the moment polytope. In order to discuss the behavior of the Lagrangians at the boundary $\del P$ of the moment polytope $P$, we need to consider the fiber at $\del P$. The symplectic structure $\omega$ can not extend to the fiber at $\del P$ since the metric $\{g_{ij}\}$ diverges at $\del P$. For this reason, we formally add the torus fibers on $\del P$ and denote by $\pi':Y'\to P$ the torus fibration. Also, for the covering space $\overline{Y'}$ of $Y'$, we use the same symbol $\pi':\overline{Y'}\to P$. We then obtain the following diagram:
\begin{equation*}
\xymatrix{
Y' \ar[rd]_{\pi'} \ar@{}[d]|{\bigcup} & & \cp  \ar[ld]^{\check{\mu}} \ar@{}[d]|{\bigcup}\\
Y \ar[rd]_\pi & P \ar@{}[d]|{\bigcup} & \cU \ar[ld]^{\check{\mu}|_\cU}\\
 & B &
}
\end{equation*}
Note that $Y'$ is no longer a symplectic manifold since no symplectic structure is defined on $\pi'^{-1}(\del P)$.

\subsection{Lagrangian multi-sections and spectral networks}
We review the notion of (tropical) Lagrangian multi-sections and spectral networks.
\begin{definition}
Let $\Sigma$ be a complete fan in $N_\R$. An \textit{$r$-fold tropical Lagrangian multi-section over $\Sigma$} is a datum $\bL^{\trop}:=\left( L^{\trop},\Sigma_{L^{\trop}},\mu,p,\varphi^{\trop} \right)$, where 
\begin{enumerate}
\item $(L^{\trop},\Sigma_{L^{\trop}})$ is a connected cone complex weighted by $\mu:L^\trop\to\Z_{>0}$,
\item $p:(L^{\trop},\Sigma_{L^\trop})\to(N_\R,\Sigma)$ is a branched covering map of a cone complex such that $\sum_{x'\in p^{-1}(x)}\mu(x')=r$ for any $x\in N_\R$,
\item $\varphi^\trop$ is a piecewise linear function on $(L^\trop,\Sigma_{L^\trop})$.
\end{enumerate}
\end{definition}
The branched covering map $p:(L^{\trop},\Sigma_{L^\trop})\to(N_\R,\Sigma)$ maps each cone $\sigma$ in $\Sigma_{L^\trop}$ homeomorphically onto a cone $p(\sigma)$ in $\Sigma$. 
We can view $\varphi^\trop$ as a multi-valued piecewise linear function on $N_\R$. In particular, a support function associated to the holomorphic line bundle over the toric variety $X_\Sigma$ is a tropical Lagrangian section (see e.g. \cite{CLS11}).\par
\begin{remark}
Historically, Payne associated branched covers together with piecewise-linear functions to toric vector bundles over a toric variety in \cite{Pay09}.
The second author studied such branched covers and piecewise-linear functions from the viewpoint of mirror symmetry and introduced the notion of tropical Lagrangian multi-sections in \cite{Sue23}.
By smoothing tropical Lagrangian multi-sections, Oh and the second author constructed Lagrangian multi-sections in \cite{OS24}, which are mirror to toric vector bundles.
For more details see \cite{Pay09,Sue23,OS24}.
\end{remark}
\begin{definition}
An immersed Lagrangian $\bL:=(i:L\to T^*\!N_\R)$ is called a \textit{Lagrangian multi-section of degree $r$} if the composition $p_\bL:= p_{N_\R} \circ i:L \to N_\R$ is an $r$-fold branched covering map. We denote by $B_\bL$ the set of branch points. If there is no confusion, we denote by the same symbol $\bL$ the image $i(L)$.
\end{definition}
We assume that all Lagrangian multi-sections in this paper satisfy a certain asymptotic condition called a \textit{$\Lambda_{\bL^\trop}$-admissibility}, which is induced by a tropical Lagrangian multi-section $\bL^\trop$. We will not go into the definition of the admissibility in this paper, but roughly speaking, it means that the behavior of the Lagrangian is controlled at infinity. 
For more details see \cite{OS24, Sue24}.  Furthermore, we assume that Lagrangian multi-sections have finitely many simple branch points. 
Namely, the monodromy around each branch point just exchanges two sheets of $p_\bL:L \to N_\R$. \par
Next, we briefly review spectral networks subordinate to Lagrangian multi-sections. Spectral networks were introduced by Gaiotto-Moore-Neitzke in \cite{GMN13}. Here, we will not review the general definition of the spectral networks subordinate to branched covering maps, but we will only review the spectral networks subordinate to Lagrangian multi-sections constructed in \cite{Sue24}.
Let $\bL$ be a Lagrangian multi-section. Fix a set of disjoint branch cuts $\{c_x\}_{x\in B_\bL}$ for $\bL$. We set
\begin{equation*}
U:=N_\R\backslash\bigsqcup_{x\in B_\bL}c_x
\end{equation*}
and write
\begin{equation*}
p_\bL^{-1}(U) = \bigsqcup_{\alpha=1}^r U^{(\alpha)}
\end{equation*}
so that $p_\bL|_{U^{(\alpha)}}:U^{(\alpha)}\to U$ is a homeomorphism. For each sheet $U^{(\alpha)}$, there exists a function $f^{(\alpha)}_\bL:U\to\R$ so that
\begin{equation*}
\bL^{(\alpha)} := i(U^{(\alpha)}) = \mathrm{graph}(df^{(\alpha)}_\bL).
\end{equation*}
Therefore, the Lagrangian multi-section is expressed as
\begin{equation*}
\bL=\bigcup_{\alpha=1}^r\bL^{(\alpha)}=\bigcup_{\alpha=1}^r\mathrm{graph}(df^{(\alpha)}_\bL).
\end{equation*}
In this sense, we regard the collection of functions $F_\bL=\{f^{(1)}_\bL,\cdots,f^{(r)}_\bL\}$ as the multi-valued function and call it the potential function of $\bL$. 
\begin{definition}
The set of rays subordinate to $\bL$ is the collection
\begin{equation*}
\cR_\bL := \{\ell:[0,a^+)\to N_\R\}
\end{equation*}
of all embedded rays so that the following conditions are satisfied.
\begin{enumerate}
\item The restriction $\ell|_{\ell^{-1}(U)}$ is a maximal solution to an $(\alpha\beta)$-gradient flow equation
\begin{equation*}
\dot{\ell}(t) = \grad(f^{(\alpha)}_\bL-f^{(\beta)}_\bL)(\ell(t))
\end{equation*}
for some distinct $\alpha,\beta=1,2,\cdots,r$.
\item If $\ell|_{\ell^{-1}(U)}$ satisfies some $(\alpha\beta)$-gradient flow equation with $\ell(0)\in B_\bL$, then $\ell\in\cR_\bL$.
\item If $\ell_1,\ell_2\in\cR_\bL$ satisfy the $(\alpha\beta)$- and $(\beta\gamma)$-gradient flow equations with $\ell_1(t_1)=\ell_2(t_2)$ for some $t_1\in \ell_1^{-1}(U),\ t_2\in\ell_2^{-1}(U)$, then there exists $\ell_3\in\cR_\bL$ such that $\ell_3$ satisfies the $(\alpha\gamma)$-gradient flow equation and $\ell_3(0)=\ell_1(t_1)=\ell_2(t_2)$.
\end{enumerate}
\end{definition}
Applying the Legendre transform $d\phi:N_\R\to B$ to each ray $\ell\in\cR_\bL$, we can get the collection of embedded rays in the moment polytope.
\begin{definition}
The \textit{spectral network subordinate to $\bL$} is the collection
\begin{equation*}
\cW_\bL := \{w:=d\phi\circ\ell:[0,a^+)\to B\ |\ \ell\in\cR_\bL\}
\end{equation*}
of all embedded rays. If a ray $\ell\in\cR_\bL$ satisfies the $(\alpha\beta)$-gradient flow equation, then $w=d\phi\circ\ell\in\cW_\bL$ is called an \textit{$(\alpha\beta)$-wall}. $\cW_\bL$ is said to be \textit{non-degenerate} if each wall has at most one branch point. The set
\begin{equation*}
|\cW_{\bL}| := \bigcup_{w\in\cW_{\bL}}w([0,a^+))
\end{equation*}
is called the \textit{support} of $\cW_{\bL}$.
\end{definition}
To ensure the non-degeneracy of the spectral network, we assume the following condition on the Lagrangian multi-sections.
\begin{definition}
We say $\bL$ is \textit{well-behaved} if $\cW_\bL$ is finite and for any wall $w\in\cW_\bL$ the limit
\begin{equation*}
\lim_{t\to a^+}w(t)
\end{equation*}
exists on $\del P$.
\end{definition}
Well-behavedness prohibits gradient lines from connecting two branch points and spinning around.
\begin{theorem}\cite[Theorem 2.14]{Sue24}
If $\bL$ is a well-behaved $\Lambda_{\bL^{\trop}}$-admissible Lagrangian multi-section, then $\cW_\bL$ is a non-degenerate spectral network.
\end{theorem}

\section{The category of holomorphic vector bundles.}\label{complex side}
In this section, we discuss the complex side of homological mirror symmetry.\par
We define a dg category $\DG^\vect(X_\Sigma)$ consisting of holomorphic vector bundles on the toric manifold $X_\Sigma$. For a holomorphic vector bundle $V$ on $X_\Sigma$, we take a connection $D$, which defines a holomorphic structure on $V$. We set the objects of $\DG^\vect(X_\Sigma)$ as holomorphic vector bundles $(V,D)$ on $X_\Sigma$. The space $\DG^\vect(X_\Sigma)((V_a,D_a),(V_b,D_b))$ of morphisms is defined as the $\Z$-graded vector space whose degree $r$ part is given by
\begin{equation*}
\DG^\vect(X_\Sigma)^r((V_a,D_a),(V_b,D_b)) := \Gamma(V_a,V_b)\otimes_{C^\infty(X_\Sigma)}\Omega^{0,r}(X_\Sigma),
\end{equation*}
where $\Gamma(V_a,V_b)$ is the space of smooth bundle maps from $V_a$ to $V_b$. 
We define a linear map $d_{ab}:\DG^\vect(X_\Sigma)^r((V_a,D_a),(V_b,D_b))\to \DG^\vect(X_\Sigma)^{r+1}((V_a,D_a),(V_b.D_b))$ as follows. We decompose $D_a$ into its holomorphic part and anti-holomorphic part $D_a= D_a^{(1,0)} + D_a^{(0,1)}$, and set $d_a:=2D_a^{(0,1)}$. Then, for $\psi_{ab}\in \DG^\vect(X_\Sigma)^r((V_a,D_a),(V_b,D_b))$, we set
\begin{equation*}
d_{ab}(\psi_{ab}) := d_b\psi_{ab} -(-1)^r\psi_{ab}d_a.
\end{equation*}
This linear map $d_{ab}:\DG^\vect(X_\Sigma)^r((V_a,D_a),(V_b.D_b))\to \DG^\vect(X_\Sigma)^{r+1}((V_a,D_a),(V_b,D_b))$ satisfies $d_{ab}^2=0$ since $(V_a,D_a)$ and $(V_b,D_b)$ are holomorphic vector bundles. The product structure 
is defined by the composition of bundle maps together with the wedge product for the anti-holomorphic differential forms.
Now, we take a full strongly exceptional collection 
\begin{equation}\label{excep coll}
\cE := \left( \cO(1), T_\cp, \cO(2) \right)
\end{equation}
where $T_\cp$ is the holomorphic tangent bundle.
\begin{remark}
We can get the exceptional collection $\cE$ by applying a right-mutation $R_1$ of the exceptional collection to $(\cO,\cO(1),\cO(2))$.
\end{remark}
We consider the full subcategory $\DG^\vect_\cE(\cp)\subset \DG^\vect(\cp)$ consisting of the exceptional collection $\cE$. The morphisms in $\DG^\vect_\cE(\cp)$ are as follows:
\begin{equation*}
\xymatrix{
\cO(1) \ar@(dl, dr)_{\id_{\cO(1)}} \ar@/^18pt/[rr]^{F_1}  \ar[rr]^{F_2} \ar@/_18pt/[rr]^{F_3} & &  T_\cp \ar@(dl, dr)_{\id_{T_\cp}} \ar@/^18pt/[rr]^{G_1}  \ar[rr]^{G_2} \ar@/_18pt/[rr]^{G_3} & &  \cO(2) \ar@(dl, dr)_{\id_{\cO(2)}}
}
\end{equation*}
Each morphism is expressed locally as 
\begin{align}\label{morphisms}
F_1 &= \begin{pmatrix}1\\0\end{pmatrix}, & F_2 &= \begin{pmatrix}0\\1\end{pmatrix}, & F_3 &= \begin{pmatrix}u\\v\end{pmatrix},\notag\\ G_1 &= \begin{pmatrix}0 & 1\end{pmatrix}, & G_2 &= \begin{pmatrix}-1 & 0\end{pmatrix}, & G_3 &= \begin{pmatrix}-v&u\end{pmatrix},
\end{align}
where $F_i,G_j$ are the bundle maps. By this expression, the compositions are given by
\begin{align}\label{rel in DGvect}
G_1 \circ F_1 &= 0, & G_2 \circ F_1 &= -1, &  G_1 \circ F_2 &= 1,\notag\\
G_2 \circ F_2 &= 0, & G_3 \circ F_2 &= u, & G_2 \circ F_3 &= -u,\\
G_3 \circ F_3 &= 0, & G_3 \circ F_1 &= -v, & G_1 \circ F_3 &= v.\notag
\end{align}

\section{Multi-valued Morse homotopy}\label{symplectic side}
In this section, we discuss the symplectic side of the homological mirror symmetry. In subsection \ref{mult-morse}, we define a category $Mo^{\mult}(P)$ of multi-valued Morse homotopy, which is a generalization of the category $\Mo(P)$ of Morse homotopy introduced in \cite{FK21}. In subsection \ref{correspondence between objects}, we review the Lagrangian sections (resp. multi-sections) corresponding to the holomorphic line (resp. vector) bundles over $\cp$. In subsection \ref{morphisms in mult Mo(P)}, we compute the morphisms in the full subcategory $\Mo^\mult_\cE(P)\subset \Mo^\mult(P)$ consisting of the collection $\cE$ of mirror objects. In subsection \ref{proof of main thm}, we prove the main theorem (Theorem \ref{main theorem}).

\subsection{The category $\Mo^{\mult}(P)$ of multi-valued Morse homotopy}\label{mult-morse}
We define a category $\Mo^{\mult}(P)$ of multi-valued Morse homotopy for a toric surface. This category is a generalization of the weighted Fukaya-Oh category defined by Kontsevich-Soibelman in \cite{KS01} to Lagrangian multi-sections. Note that, we impose the same boundary condition on $Mo^\mult(P)$ as on $Mo(P)$ defined by Futaki-Kajiura in \cite{FK21}, but  we allow only transverse intersections and exclude clean intersections.\par
The objects of $\Mo^{\mult}(P)$ are Lagrangian multi-sections $\bL$ of $p_{N_\R}: Y \to N_\R$ whose lifts to $\overline{Y}$ are well-behaved $\Lambda_{\bL^{\trop}}$-admissible Lagrangian multi-sections.
Given two objects $\bL_1= \cup_{i=1}^{r_1}\bL_1^{(i)},\ \bL_2 = \cup_{\alpha=1}^{r_2}\bL_2^{(\alpha)} \in \Mo^{\mult}(P)$, 
we assume that $\bL_1$ and $\bL_2$ intersect transversely, avoiding each other's ramification locus.
The space $\Mo^\mult(P)(\bL_1,\bL_2)$ is the $\Z$-grading vector space spanned by the points $v\in\pi'(\bL_1 \cap \bL_2)\subset P$ which satisfy the following condition:
\begin{description}
\item[(M)] For the stable manifold $S_v$ of the gradient vector field $\grad(f_{\bL_2}^{(\alpha)} - f_{\bL_1}^{(i)})$ at $v\in\pi'(\bL_1^{(i)}\cap \bL_2^{(\alpha)})$, the point $v$ is an interior point of $S_v\cap P\subset S_v$.
\end{description}
Then we define the degree of $v$ by $|v|:=\dim(S_v)$. The space $\Mo^\mult(P)(\bL_1,\bL_2)$ is written as
\begin{equation*}
\Mo^\mult(P)(\bL_1,\bL_2) := \bigoplus_{\substack{p\in\pi'(\bL_1 \cap \bL_2)\\%
\text{$p$ satisfies }\mathbf{(M)}}}\C\cdot p.
\end{equation*}
In particular, for a Lagrangian section $L$, the space $\Mo^\mult(P)(L,L)$ is spanned by $P$ itself which is of degree zero.
For an $r$-fold Lagrangian multi-section $\bL$, the space $\Mo^\mult(P)(\bL,\bL)$ is defined by
\begin{equation*}
\Mo^\mult(P)(\bL,\bL) := \bigoplus_{i=1}^r \C \cdot  P^{(i)} \oplus \bigoplus_{\substack{i < j \\ b^{(ij)}\in B_\bL}}\C \cdot  b^{(ij)} \oplus \bigoplus_{\substack{i\neq j \\ x^{(ij)}\in \pi'(\bL^{(i)}\cap \bL^{(j)})\backslash B_\bL \\ \text{$x^{(ij)}$ satisfies }\mathbf{(M)}}} \C \cdot  x^{(ij)}
\end{equation*}
where $P^{(i)}$ is a copy of the moment polytope, $b^{(ij)}$ is a branch point interchanging the $i$-th sheet and the $j$-th sheet, and $x^{(ij)}$ is a projection of the self-intersection to the moment polytope. The degree of each $P^{(i)}$ is zero, i.e., $|P^{(i)}|=0$. 
We formally set the degree of $b^{(ij)}$ as one, i.e, $|b^{(ij)}|=1$. The degree of $x^{(ij)}$ is defined by the dimension of the stable manifold $S_{x^{(ij)}}$ of the gradient vector field $\grad(f_\bL^{(j)}-f_\bL^{(i)})$ at $x^{(ij)}$. 
Note that $b^{(ij)}$ and $b^{(ji)}$ are identified, but $x^{(ij)}$ and $x^{(ji)}$ are distinct. These generators $x^{(ij)}$ correspond to the self-intersections in immersed Lagrangian Floer theory \cite{AJ}.\par
\begin{remark}
Condition {\bf (M)} is the boundary condition. If $v \in \pi'(\bL_1^{(i)}\cap \bL_2^{(\alpha)})$ is in $B=\mathrm{Int}(P)$, the point $v$ satisfies Condition {\bf(M)} automatically. Then the point $v$ is the critical point of $f_{\bL_2}^{(\alpha)}-f_{\bL_1}^{(i)}$ and the degree $|v|$ coincides with the Morse index. If $v \in \pi'(\bL_1^{(i)}\cap \bL_2^{(\alpha)})$ is in $\del P$, then it is nontrivial whether Condition {\bf(M)} is satisfied.
Roughly speaking, the space of morphisms is spanned by the critical points of the multi-valued function $F_{\bL_2}-F_{\bL_1}$.
\end{remark}
In general, $\Mo^{\mult}(P)$ must be defined as an $A_\infty$-category, but the higher products $\fm_k\ (k\geq3)$ of the cohomology $H(\Mo^\mult_\cE(P))$ of the full subcategory $\Mo^{\mult}_{\cE}(P)$ considered in this paper vanish. This is why we only explain $\fm_1$ and $\fm_2$ here.\par
For two Lagrangians $\bL_1= \cup_{i=1}^{r_1}\bL_1^{(i)},\bL_2 = \cup_{\alpha=1}^{r_2}\bL_2^{(\alpha)}$ and two points $p,q\subseteq \pi'(\bL_1\cap \bL_2)$, we consider the contribution of spectral networks to gradient lines.
\begin{definition}
Let $p\in\mathrm{Crit}(f^{(\alpha)}_{\bL_2}-f^{(i)}_{\bL_1})$ and $q\in\mathrm{Crit}(f^{(\beta)}_{\bL_2}-f^{(j)}_{\bL_1})$ be the generators of  $\Mo^{\mult}(P)(\bL_1,\bL_2)$ for $\bL_1=\cup_{k=1}^{r_1}\bL_1^{(k)}$ and $\bL_2=\cup_{\gamma=1}^{r_2}\bL_2^{(\gamma)}$.
Let $\ell:\R\to P$ be an embedded line with marked points $\{\ell(t_1),\cdots,\ell(t_n)\}$ for $-\infty<t_1<t_2<\cdots<t_n<\infty$.
The line $(\ell, \{t_i\}_{i=1}^{n})$ is called a \textit{jagged gradient line from $p$ to $q$} if $(\ell, \{t_i\}_{i=1}^{n})$ satisfies the following conditions.
\begin{enumerate}
\item $\lim_{t\to-\infty}\ell(t)=p,\ \lim_{t\to\infty}\ell(t)=q$.
\item $\{\ell(t_1),\cdots,\ell(t_n)\} = \ell(\R)\cap\left(|\cW_{\bL_1}|\cup|\cW_{\bL_2}|\cup\{c_x\}_{x\in B_{\bL_1}\cup B_{\bL_2}}\right)$
\item The restriction $\ell|_{(-\infty,t_1)}$ satisfies the gradient flow equation
\begin{equation*}
\dot{\ell}|_{(-\infty,t_1)}(t) = \grad\left( f^{(\alpha)}_{\bL_2}-f^{(i)}_{\bL_1} \right)(\ell(t)).
\end{equation*}
\item The restriction $\ell|_{(t_n,\infty)}$ satisfies the gradient flow equation
\begin{equation*}
\dot{\ell}|_{(t_n,\infty)}(t) = \grad\left( f^{(\beta)}_{\bL_2}-f^{(j)}_{\bL_1} \right)(\ell(t)).
\end{equation*}
\item If $\ell|_{(t_{m-1},t_m)}$ satisfies the gradient flow equation
\begin{equation*}
\dot{\ell}|_{(t_{m-1},t_m)}(t) = \grad\left( f^{(\gamma)}_{\bL_2}-f^{(k)}_{\bL_1} \right)(\ell(t))
\end{equation*}
and $\ell(t_m)$ lies on the $(kk')$-wall of $\cW_{\bL_1}$ or the $(kk')$-branch cut of $\bL_1$ (resp. the $(\gamma'\gamma)$-wall of $\cW_{\bL_2}$ or the $(\gamma'\gamma)$-branch cut of $\bL_2$), then $\ell|_{(t_m,t_{m+1})}$ satisfies the gradient flow equation
\begin{equation*}
\dot{\ell}|_{(t_{m},t_{m+1})}(t) = \grad\left( f^{(\gamma)}_{\bL_2}-f^{(k')}_{\bL_1} \right)(\ell(t))\ \ \ \left(\text{resp.}\ \ \ =\grad\left( f^{(\gamma')}_{\bL_2}-f^{(k)}_{\bL_1} \right)(\ell(t))\right).
\end{equation*}
\end{enumerate}
An ordinary gradient line is one that has no marked points, i.e., the case where $n=0$.
Figure \ref{fig: jagged gradient line} shows the interactions between the gradient lines and the spectral networks . We call the jagged gradient line $(\ell, \{t_i\}_{i=1}^{n})$ \textit{generic} if $\ell(\R)\cap|\cW_{\bL_1}|\cap|\cW_{\bL_2}|=\emptyset$. Namely, the generic jagged gradient line $(\ell, \{t_i\}_{i=1}^{n})$ interacts with the spectral networks $\cW_{\bL_1}, \cW_{\bL_2}$ one at a time, rather than simultaneously.
\begin{figure}[h]
\center
\begin{tikzpicture}[scale=1.2]
\draw(0,0)--  (3.5,0)node[right]{grad$(f_{\bL_2}^{(2)}-f_{\bL_2}^{(1)})$};
\draw(0,0)--  (-1.5,3)node[above]{grad$(f_{\bL_2}^{(1)}-f_{\bL_2}^{(2)})$};
\draw(0,0)--  (-0.5,-1.5)node[below]{grad$(f_{\bL_2}^{(1)}-f_{\bL_2}^{(2)})$};
\draw[arrows = {-Stealth[scale=1.5]}](0,0)--(2,0);
\draw[arrows = {-Stealth[scale=1.5]}](0,0)--(-0.25,0.5);
\draw[arrows = {-Stealth[scale=1.5]}](0,0)--  (-0.25,-0.75);
\draw[dashed,orange](0,0) -- (-3.5,0) node[left]{$(12)$-branch cut of $\bL_2$};
\draw[red](-2,1)--  (-1.25,2.5);
\draw[red](-1.25,2.5) -- (0,3);
\draw[blue](-2,1)--  (1,2);
\draw[blue](-2,1)--  (-1,0.5);
\draw[red](-2,1)--  (-1.25,-1);
\draw[red](-2,1)--  (-0.5,1);
\draw[red](-0.5,1)--  (1,0);
\draw[red](1,0)--  (1.5,-1);
\draw[arrows = {-Stealth[scale=1.5]},red](-2,1)--  (-1.5,2);
\draw[arrows = {-Stealth[scale=1.5]},blue](-2,1)--  (-1.25,1.25);
\draw[arrows = {-Stealth[scale=1.5]},blue](-2,1)--  (0.4,1.8);
\draw[arrows = {-Stealth[scale=1.5]},red](-1.25,2.5)--  (-0.5,2.8);
\draw[arrows = {-Stealth[scale=1.5]},blue](-2,1)--  (-1.3,0.65);
\draw[arrows = {-Stealth[scale=1.5]},red](-2,1)--  (-1.72,0.25);
\draw[arrows = {-Stealth[scale=1.5]},red](-2,1)--  (-1.4375,-0.5);
\draw[arrows = {-Stealth[scale=1.5]},red](-2,1)--  (-1,1);
\draw[arrows = {-Stealth[scale=1.5]},red](-0.5,1)--  (0.25,0.5);
\draw[arrows = {-Stealth[scale=1.5]},red](1,0)--  (1.25,-0.5);
\draw(-1.75,1.6)node[left]{grad$(f_{\bL_2}^{(2)}-f_{\bL_1})$};
\draw(0,1.4)node[right]{grad$(f_{\bL_2}^{(2)}-f_{\bL_1})$};
\draw(-0.75,2.55)node[right]{grad$(f_{\bL_2}^{(1)}-f_{\bL_1})$};
\draw(-1.4375,-0.5)node[left]{grad$(f_{\bL_2}^{(1)}-f_{\bL_1})$};
\draw(1.25,-0.5)node[right]{grad$(f_{\bL_2}^{(1)}-f_{\bL_1})$};
\fill[black](-2,1)circle(0.06) node[left]{$p$};
\fill[black](0,3)circle(0.06) node[right]{$q_1$};
\fill[black](1,2)circle(0.06) node[right]{$q_2$};
\fill[black](-1,0.5)circle(0.06) node[right]{$q_3$};
\fill[black](-1.25,-1)circle(0.06) node[left]{$q_4$};
\fill[black](1.5,-1)circle(0.06) node[right]{$q_5$};
\draw(3,2)node[right]{$\ p,q_2,q_3\in\mathrm{Crit}(f_{\bL_2}^{(2)}-f_{\bL_1})$};
\draw(3,3)node[right]{$q_1,q_4,q_5\in\mathrm{Crit}(f_{\bL_2}^{(1)}-f_{\bL_1})$};
\end{tikzpicture}
\caption{For a Lagrangian section $\bL_1=\mathrm{graph}(df_{\bL_1})$ and a Lagrangian multi-section $\bL_2=\mathrm{graph}(df_{\bL_2}^{(1)})\cup\mathrm{graph}(df_{\bL_2}^{(2)})$, the black lines are the spectral network $\cW_{\bL_2}$ and the orange line is the branch cut of $\bL_2$. The red lines are jagged gradient lines from $p$ to $q_1,q_4,q_5$. The blue lines from $p$ to $q_2,q_3$ are ordinary gradient lines, which do not interact with the $(12)$-wall of $\cW_{\bL_2}$.}
\label{fig: jagged gradient line}
\end{figure}
\end{definition}
\begin{remark}
The label of the wall is important because the gradient line may not interact with that wall. The gradient line of $f^{(\alpha)}_{\bL_2}-f^{(i)}_{\bL_1}$ does not interact with $(ji)$-walls of $\cW_{\bL_1}$ and $(\alpha\beta)$-walls of $\cW_{\bL_2}$. 
\end{remark}
\begin{remark}
Note that the definition of jagged gradient lines includes the interaction with the branch cut. This interaction is due to the continuity of the multi-valued function and is different from the interaction with the walls of the spectral network. The interaction with the wall corresponds to a pseudo-holomorphic disc containing a ramification locus in Floer theory. For more details see \cite{Sue24,Fuk05,Nho23}.
\end{remark}
We denote by $\cMj(p,q)$ the moduli space of jagged gradient lines from $p$ to $q$. For the case where $|q|=|p|+1$, we assume that $\cMj(p,q)$ is a finite set. Then, we define a linear map $\fm_1:\Mo^{\mult}(P)(\bL_1,\bL_2)\rightarrow \Mo^{\mult}(P)(\bL_1,\bL_2)$ by
\begin{equation*}
\fm_1(p) := \sum_{\substack{q\in \Mo^{\mult}(P)(\bL_1,\bL_2)\\|q|=|p|+1}}  \sum_{\ell\in\cMj(p,q)} \pm e^{-A(\ell)}q
\end{equation*}
where $p,q$ are the generators of $\Mo^{\mult}(P)(\bL_1,\bL_2)$ and $A(\ell)\in[0,\infty]$ is the symplectic area of the piecewise smooth disc in $\pi'^{-1}(\ell(\R))$.\par
Secondly, we take three Lagrangians $\bL_1,\bL_2,\bL_3$ and three points $v_{12}\subseteq \pi'(\bL_1\cap \bL_2)$, $v_{23}\subseteq\pi'(\bL_2\cap \bL_3)$, $v_{13}\subseteq\pi'(\bL_1\cap \bL_3)$. Let $\cMj(v_{12},v_{23};v_{13})$ be the moduli space of the trivalent gradient trees starting at $v_{12}$, $v_{23}$ and ending at $v_{13}$ such that one or more of its edges are jagged gradient lines.
\begin{figure}[h]
\center
\begin{tikzpicture}[scale=1.2]
\draw(-2,0)--  (-1,1);
\draw(-2,0)--  (-3,1);
\draw(-2,0)--  (-2,-1);
\draw(0,0) node{$\longmapsto$};
\draw(0,0.3) node{$\gamma$};
\draw(4,0)--  (5,1) node[above]{$v_{23}$};
\draw(4,0)--  (3,1) node[above]{$v_{12}$};
\draw(4,0)--  (4,-1) node[below]{$v_{13}$};
\draw[arrows = {-Stealth[scale=1.5]}](5,1)--  (4.5,0.5);
\draw(4.7,0.5) node[right]{$\grad(f_3-f_2)$};
\draw[arrows = {-Stealth[scale=1.5]}](3,1)--  (3.5,0.5);
\draw(3.3,0.5)node[left]{$\grad(f_2-f_1)$};
\draw[arrows = {-Stealth[scale=1.5]}](4,0)--  (4,-0.5);
\draw(4.2,-0.5)node[right]{$\grad(f_3-f_1)$};
\end{tikzpicture}
\caption{An ordinary trivalent gradient tree. The elements of $\cM_{\mathrm{jag}}(v_{12},v_{23};v_{13})$ are obtained by replacing the edges with jagged gradient lines.}
\label{fig: tree}
\end{figure}
For the case where $|v_{13}|=|v_{12}|+|v_{23}|$, we assume that $\cMj(v_{12},v_{23};v_{13})$ is a finite set.
Then, we define the multilinear product $\fm_2:\Mo^{\mult}(P)(\bL_1,\bL_2)\otimes \Mo^{\mult}(P)(\bL_2,\bL_3) \to \Mo^{\mult}(P)(\bL_1,\bL_3)$ by
\begin{equation*}
\fm_2(v_{12},v_{23}) := \sum_{\substack{v_{13}\in \Mo^{\mult}(P)(\bL_1,\bL_3)\\|v_{13}|=|v_{12}|+|v_{23}|}}\sum_{\gamma\in \cMj(v_{12},v_{23};v_{13})}\pm e^{-A(\gamma)}v_{13},
\end{equation*}
where $v_{ij}$ are the generators of $\Mo^{\mult}(P)(\bL_i,\bL_j)$ and $A(\gamma)\in[0,\infty]$ is the symplectic area of the piecewise smooth disc.\par
Similarly, the higher products $\fm_n\ (n\geq3)$ are defined by counting the ordinary gradient trees and counting the gradient trees that include jagged gradient lines as edges.

\subsection{Lagrangian sections and Lagrangian multi-sections}\label{correspondence between objects}
Firstly, we consider the Lagrangian sections $L(k)$ corresponding to the holomorphic line bundle $\cO(k)$. These Lagrangian sections are defined by
\begin{equation*}
\begin{pmatrix}
y^1\\y^2
\end{pmatrix}
=
k\begin{pmatrix}
\dfrac{e^{2\xi_1}}{1+e^{2\xi_1}+e^{2\xi_2}}\\\dfrac{e^{2\xi_2}}{1+e^{2\xi_1}+e^{2\xi_2}}
\end{pmatrix}
=
k\begin{pmatrix}
x^1\\x^2
\end{pmatrix},
\end{equation*}
since the connection form of $\cO(k)$ is given by
\begin{equation*}
d - k\dfrac{\bar{u}du+\bar{v}dv}{1+u\bar{u}+v\bar{v}} = d - k\dfrac{e^{2\xi_1}(d\xi_1+\rt dy_1)+e^{2\xi_2}(d\xi_2+\rt dy_2)}{1+e^{2\xi_1}+e^{2\xi_2}}.
\end{equation*}
The potential function $f_{L(k)}$ of $L(k)$ is given by
\begin{equation*}
f_{L(k)} = \frac{k}{2}\log(1+e^{2\xi_1}+e^{2\xi_2}) = -\frac{k}{2}\log(1-x^1-x^2).
\end{equation*}
The potential function means that the Lagrangian section $L(k)$ is the graph of the $1$-form $df_{L(k)}$ in $T^*\!N_\R$. The corresponding gradient vector field is
\begin{equation*}
\grad(f_{L(k)}) = \sum_{i,j=1}^2\frac{\del f_{L(k)}}{\del x^i}g^{ij}\frac{\del}{\del x^j} = \frac{\del f_{L(k)}}{\del \xi_1}\frac{\del}{\del x^1} + \frac{\del f_{L(k)}}{\del \xi_2}\frac{\del}{\del x^2} = k\left( x^1\frac{\del}{\del x^1} + x^2\frac{\del}{\del x^2} \right).
\end{equation*}
\par
Secondly, we consider the Lagrangian multi-section $\bL_{T_\cp}$ corresponding to the holomorphic tangent bundle $T_\cp$. This Lagrangian multi-section constructed in \cite{OS24} is realized by the embedded Lagrangian in $Y$.
We briefly review the construction of Lagrangian multi-sections from tropical Lagrangian multi-sections.
Firstly, for a tropical Lagrangian multi-section $(\bL^{\trop},\varphi^{\trop})$, we construct an embedded Lagrangian near the infinity by smoothing the piecewise linear function $\varphi^\trop$.
Secondly, we construct a local model of the Lagrangian multi-section around the branch points. We review this part in appendix \ref{local model}.
Finally, we obtain an embedded Lagrangian multi-section by gluing the outer part and the local model. The resulting Lagrangian satisfies the asymptotic condition by the construction.
In this paper, we apply the construction to the tropical Lagrangian multi-section $(\bL^{\trop}_{T_\cp},\varphi^{\trop}_{T_\cp})$ obtained from the holomorphic tangent bundle $T_\cp$.
Figure \ref{fig:Tcp2} shows the tropical Lagrangian multi-section $(\bL^{\trop}_{T_\cp},\varphi^{\trop}_{T_\cp})$ associated to $T_\cp$.
\begin{figure}[h]
		\centering
\begin{tikzpicture}[scale=0.8]
\draw(3,0) node[right][black]{$\ \ \ \xrightarrow{\ \ \ p\ \ \ }$}  --  (0,0);
\draw(0,0) --  (-1.5,-2.6) ;
\draw(-1.5,2.6)  --  (0,0);
\draw(0,0) --  (-3,0);
\draw(1.5,2.6)  --  (0,0);
\draw(0,0) --  (1.5,-2.6);
\draw(0.1,0.9) node{$\sigma_0^{(1)}$};
\draw(-0.8,0.4) node{$\sigma_1^{(1)}$};
\draw(-0.8,-0.4) node{$\sigma_2^{(1)}$};
\draw(0,-0.9) node{$\sigma_0^{(2)}$};
\draw(0.8,-0.4) node{$\sigma_1^{(2)}$};
\draw(0.8,0.4) node{$\sigma_2^{(2)}$};
\draw(0,2) node{$-\xi_1$};
\draw(-2,1) node{$\xi_1$};
\draw(-2,-1) node{$\xi_2$};
\draw(0,-2) node{$-\xi_2$};
\draw(2,-1) node{$\xi_1-\xi_2$};
\draw(2,1) node{$-\xi_1+\xi_2$};
\end{tikzpicture}
\begin{tikzpicture}[scale=0.9]
\draw(0,0) --  (-2,0) node[left]{$\rho_1$};
\draw(0,0) --  (0,-2) node[below]{$\rho_2$};
\draw(0,0) --  (1.4,1.4) node[above right]{$\rho_0$};
\draw(-1,-1) node{$\sigma_{0}$};
\draw(1,-1) node{$\sigma_{1}$};
\draw(-1,1) node{$\sigma_{2}$};
\end{tikzpicture}
	    \caption{The tropical Lagrangian multi-section $(\bL_{T_{\cp}}^{\trop},\varphi^{\trop}_{T_\cp})$ associated to the holomorphic tangent bundle $T_\cp$.}
		\label{fig:Tcp2}
\end{figure}
Let $\bL_{T_\cp}$ be the Lagrangian multi-section constructed from $(\bL^{\trop}_{T_\cp},\varphi^{\trop}_{T_\cp})$ and $F_{\bL_{T_\cp}}=\{f^{(1)}_{\bL_{T_\cp}},f^{(2)}_{\bL_{T_\cp}}\}$ be the potential function of $\bL_{T_\cp}$. We denote by $\bL_{T_\cp}^{(1)}$ (resp. $\bL_{T_\cp}^{(2)}$) the first (resp. second) sheet of $\bL_{T_\cp}$. 
For a maximal cone $\sigma$ of $\Sigma$, the fiber coordinates of $\bL_{T_\cp}$ at the vertex $\check{\sigma}$ of the moment polytope are given by $d(\varphi^{\trop}_{T_\cp}|_{\sigma^{(1)}})$ and $d(\varphi^{\trop}_{T_\cp}|_{\sigma^{(2)}})$ since the potential function $F_{\bL_{T_\cp}}$ is the smoothing of $\varphi^{\trop}_{T_\cp}$.
The fiber coordinate $(y^1,y^2)$ of $\bL_{T_\cp}$ at each vertex of the moment polytope is depicted in Figure \ref{fig:fib coord of multi-sec}.
\begin{figure}[h]
		\centering
\begin{tikzpicture}[scale=1.3]
\draw(-4,0)  --  (-1,0) -- (-4,3) --cycle;
\draw(4,0)  --  (1,0) -- (1,3)  --cycle;
\draw[dashed,orange] (-4,1.75) to[out=350,in=120] (-3,1);
\draw[dashed,orange] (1,1.75) to[out=340,in=120] (2,1);
\fill[black](-4,0)circle(0.06);
\fill[black](-1,0)circle(0.06);
\fill[black](-4,3)circle(0.06);
\fill[black](1,0)circle(0.06);
\fill[black](4,0)circle(0.06);
\fill[black](1,3)circle(0.06);
\draw(-4,0) node[above right]{$\check{\sigma}_0$};
\draw(-1.3,0) node[above left]{$\check{\sigma}_1$};
\draw(-4,2.7) node[below right]{$\check{\sigma}_2$};
\draw(1,0) node[above right]{$\check{\sigma}_0$};
\draw(3.7,0) node[above left]{$\check{\sigma}_1$};
\draw(1,2.7) node[below right]{$\check{\sigma}_2$};
\draw(-4,0) node[left]{$\begin{pmatrix}-1\\0\end{pmatrix}$};
\draw(-1,0) node[right]{$\begin{pmatrix}1\\0\end{pmatrix}$};
\draw(-4,3) node[left]{$\begin{pmatrix}0\\1\end{pmatrix}$};
\draw(1,0) node[left]{$\begin{pmatrix}0\\-1\end{pmatrix}$};
\draw(4,0) node[right]{$\begin{pmatrix}1\\-1\end{pmatrix}$};
\draw(1,3) node[left]{$\begin{pmatrix}-1\\1\end{pmatrix}$};
\end{tikzpicture}
	    \caption{The fiber coordinate of $\bL_{T_\cp}$ at each vertex of the moment polytope. The orange dashed lines are the branch cuts.  The left side is the first sheet $\bL_{T_\cp}^{(1)}$ and the right side is the second sheet $\bL_{T_\cp}^{(2)}$ of $\bL_{T_\cp}=\bL_{T_\cp}^{(1)}\cup\bL_{T_\cp}^{(2)}$.}
		\label{fig:fib coord of multi-sec}
\end{figure}
We can take a local model around the branch point of $\bL_{T_\cp}$ so that $L(1)\cap\bL_{T_\cp}=\emptyset$ and $\bL_{T_\cp}\cap L(2)=\emptyset$ in the total space $Y=\pi^{-1}(B)$.
Furthermore, in a neighborhood of each edge of the moment polytope, $\bL_{T_\cp}$ behaves like $L(1)$ or $L(2)$ asymptotically since $\bL_{T_\cp}$ is defined by the smoothing of $\varphi^{\trop}_{T_\cp}$. To avoid clean intersections, we take the realization of $\bL_{T_\cp}$ in the neighborhood of $\del P$ as shown in Figure \ref{fig:pLag} symmetrically. Such perturbation can be realized as a Hamiltonian diffeomorphism and does not change the mirror vector bundle since they do not change the boundary conditions at each vertex (see \cite{Cha09, OS24} for details).
\begin{figure}[h]
\center
\begin{tikzpicture}
\draw[dashed](-2,0) node[left][black]{0} --  (2,0)node[right][black]{$0$};
\draw[dashed](-2,4) node[left][black]{1}  --  (2,4) node[right][black]{$1$};
\draw[thick] (-2,0) -- (2,4);
\draw[thick] (-2,0) -- (0,4);
\draw[thick] (0,0) -- (2,4);
\draw[thick][blue] (-2,0) to[out=35,in=235] (0,2) to[out=55,in=215] (2,4);
\draw[thick][red] (-2,0) to[out=70,in=230] (0,4);
\draw[thick][red] (0,0) to[out=50,in=250] (2,4);
\end{tikzpicture}
\caption{The realization of $\bL_{T_\cp}$ on the edge $\check{\rho}_2$. The black lines are the Lagrangian sections $L(1)$ and $L(2)$. The red line is the first sheet $\bL_{T_\cp}^{(1)}$ and the blue line is the second sheet $\bL_{T_\cp}^{(2)}$ of $\bL_{T_\cp}=\bL_{T_\cp}^{(1)}\cup\bL_{T_\cp}^{(2)}$.}
\label{fig:pLag}
\end{figure}
\par 
We denote by $\cE$ the collection $(L(1),\bL_{T_\cp},L(2))$ of the Lagrangians which is the same notation as the exceptional collection given in subsection \ref{complex side}.

\subsection{The space of morphisms in $Mo^{\mult}_\cE(P)$}\label{morphisms in mult Mo(P)}
We compute the space of morphisms in the full subcategory $\Mo^{\mult}_\cE(P)$ consisting of $\cE$ under the above assumption. For the Lagrangians $\bL_1,\bL_2 \in \Mo^\mult(P)$, let $(y^1_{\bL_1},y^2_{\bL_2})(x),(y^1_{\bL_1},y^2_{\bL_2})(x)$ be the fiber coordinates of $\bL_1,\bL_2$ respectively. Then, the intersections $\bL_1\cap\bL_2$ are expressed as
\begin{equation*}
\begin{pmatrix}
y^1_{\bL_1}\\
y^2_{\bL_1}
\end{pmatrix}(x) + \begin{pmatrix}
i_1\\
i_2
\end{pmatrix} = \begin{pmatrix}
y^1_{\bL_2}\\
y^2_{\bL_2}
\end{pmatrix}(x)
\end{equation*}
in the covering space $\pi:\overline{Y'}\to P$ where $(i_1,i_2)\in\Z^2$. If there exists a nonempty intersection for fixed $(i_1,i_2)\in\Z^2$, then we set $v_{(i_1,i_2)}$ as the projection of that intersection to the moment polytope.
We can view the Lagrangian $\bL$ as the graph of the $1$-form $dF_{\bL}$ for some potential function $F_{\bL}$. Thus the point $v_{(i_1,i_2)}$ is the critical point of the difference of the potential functions of the two Lagrangians.
The corresponding gradient vector field of $v_{(i_1,i_2)}$ is given by
\begin{equation*}
\grad\left(F_{\bL_2}-(F_{\bL_1}+i_1\xi_1+i_2\xi_2)\right) = \left(y^1_{\bL_2}-y^1_{\bL_1}-i_1\right)\frac{\del}{\del x^1} + \left(y^2_{\bL_2}-y^2_{\bL_1}-i_2\right)\frac{\del}{\del x^2},
\end{equation*}
where $F_{\bL_1}$ and $F_{\bL_2}$ are the local potential functions of $\bL_1$ and $\bL_2$. Note that since $F_{\bL_1}$ and $F_{\bL_2}$ are multi-valued functions, we need to consider each branch separately. Namely, for $F_{\bL_1}=\{f_{\bL_1}^{(1)},\cdots,f_{\bL_1}^{(r_1)}\}$ and $F_{\bL_2}=\{f_{\bL_2}^{(1)},\cdots,f_{\bL_2}^{(r_2)}\}$, the gradient vector field of $v_{(i_1,i_2)}\in\mathrm{Crit}(f_{\bL_2}^{(\alpha)}-f_{\bL_1}^{(i)})$ is given by
\begin{equation*}
\grad\left(f_{\bL_2}^{(\alpha)}-(f_{\bL_1}^{(i)}+i_1\xi_1+i_2\xi_2)\right) = \left(y^1_{\bL_2^{(\alpha)}}-y^1_{\bL_1^{(i)}}-i_1\right)\frac{\del}{\del x^1} + \left(y^2_{\bL_2^{(\alpha)}}-y^2_{\bL_1^{(i)}}-i_2\right)\frac{\del}{\del x^2},
\end{equation*}
where $\bL_1^{(i)}$ and $\bL_2^{(\alpha)}$ are the sheets of $\bL_1= \cup_{i=1}^{r_1}\bL_1^{(i)} = \cup_{i=1}^{r_1}\mathrm{graph}(df^{(i)}_{\bL_1})$ and $\bL_2 = \cup_{\alpha=1}^{r_2}\bL_2^{(\alpha)} = \cup_{i=1}^{r_2}\mathrm{graph}(df^{(\alpha)}_{\bL_2})$.
Therefore, each generator of $\Mo^\mult(P)(\bL_1,\bL_2)$ has a $\Z^2$-grading related to the lift of the Lagrangians. The multilinear map $\fm_n$ must preserve the $\Z^2$-grading. For instance, the product $\fm_2$ is given by
\begin{equation*}
\fm_2(u_{(i_1,i_2)},v_{(j_1,j_2)}) = \sum_\gamma\pm e^{-A(\gamma)} w_{(k_1,k_2)}
\end{equation*}
where $u_{(i_1,i_2)}\in\Mo^\mult(P)(\bL_1,\bL_2)$, $v_{(j_1,j_2)}\in\Mo^\mult(P)(\bL_2,\bL_3)$, and $w_{(k_1,k_2)}\in\Mo^\mult(P)(\bL_1,\bL_3)$. Preserving the $\Z^2$-grading is $(k_1,k_2)=(i_1+j_1,i_2+j_2)$, which means
\begin{equation*}
F_{\bL_3}-(F_{\bL_1}+k_1\xi_1+k_2\xi_2) = (F_{\bL_3}-(F_{\bL_2}+j_1\xi_1+j_2\xi_2))+(F_{\bL_2}-(F_{\bL_1}+i_1\xi_1+i_2\xi_2)).
\end{equation*}
This condition is related to the existence of pseudo-holomorphic disks bounded by Lagrangians. For more details see \cite{Kaj14,FK21}.\par
Based on these insights, we compute each space of morphisms in $Mo_\cE^\mult(P)$.
\begin{itemize}
\item $\Mo^\mult_\cE(P)(L(1),\bL_{T_\cp})$: \par
The projections of the intersection $\pi'(L(1)\cap\bL_{T_\cp})$ and gradient lines are depicted in Figure \ref{fig:1toT}.
\begin{figure}[h]
		\centering
\begin{tikzpicture}[scale=1.3]
\draw[dashed](-4,0)  --  (-1,0) -- (-4,3) --cycle;
\draw[dashed](4,0)  --  (1,0) -- (1,3)  --cycle;
\draw[dashed,orange] (-4,1.75) to[out=350,in=120] (-3,1);
\draw[dashed,orange] (1,1.75) to[out=340,in=120] (2,1);
\fill[black](-4,0)circle(0.06) node[left]{$u^0_{(-1,0)}$};
\fill[black](-1,0)circle(0.06) node[right]{$u^1_{(0,0)}$};
\fill[black](-4,3)circle(0.06) node[left]{$u^2_{(0,0)}$};
\fill[black](1,0)circle(0.06) node[left]{$u^0_{(0,-1)}$};
\fill[black](4,0)circle(0.06) node[right]{$u^1_{(0,-1)}$};
\fill[black](1,3)circle(0.06) node[left]{$u^2_{(-1,0)}$};
\fill[black](-4,1.5)circle(0.06) node[left]{$U_{(-1,0)}$};
\fill[black](-2.5,1.5)circle(0.06) node[above right]{$U_{(0,0)}$};
\fill[black](2.5,0)circle(0.06) node[above]{$U_{(0,-1)}$};
\draw[thick](-4,0) -- (-4,1.5);
\draw[arrows = {-Stealth[scale=1.5]}](-4,1.5) -- (-4,0.75);
\draw[thick](-4,1.75) -- (-4,1.5);
\draw[arrows = {-Stealth[scale=1.5]}](-4,1.5) -- (-4,1.75);
\draw[thick](1,3) -- (1,1.75);
\draw[arrows = {-Stealth[scale=1.5]}](1,2) -- (1,2.5);
\draw[thick](-4,3) -- (-2.5,1.5);
\draw[arrows = {-Stealth[scale=1.5]}](-2.5,1.5) -- (-3.25,2.25);
\draw[thick](-1,0) -- (-2.5,1.5);
\draw[arrows = {-Stealth[scale=1.5]}](-2.5,1.5) -- (-1.75,0.75);
\draw[thick](1,0) -- (2.5,0);
\draw[arrows = {-Stealth[scale=1.5]}](2.5,0) -- (1.75,0);
\draw[thick](4,0) -- (2.5,0);
\draw[arrows = {-Stealth[scale=1.5]}](2.5,0) -- (3.25,0);
\draw (-2.5,-0.6) node{$\pi'(L(1)\cap\bL_{T_\cp}^{(1)})$};
\draw (2.5,-0.6) node{$\pi'(L(1)\cap\bL_{T_\cp}^{(2)})$};
\end{tikzpicture}
	    \caption{The projection of the intersection $L(1)\cap\bL_{T_\cp}$ to the moment polytope and gradient trajectories.}
		\label{fig:1toT}
\end{figure}
By the behavior of $\bL_{T_\cp}$ and $L(1)$ over each edge of the moment polytope, each $U_{(i_1,i_2)}$ forms a generator whose degree turns out to be $|U_{(i_1,i_2)}|=0$. On the other hand, the other intersections do not form generators since each $u^k_{(i_1,i_2)}$ does not satisfy Condition {\bf (M)}.
Hence we have
\begin{equation*}
\begin{split}
H^0(\Mo^\mult_\cE(P))(L(1),\bL_{T_\cp}) &\cong \Mo^\mult_\cE(P)(L(1),\bL_{T_\cp})\\
&\cong \C\cdot U_{(0,0)} \oplus \C \cdot U_{(-1,0)} \oplus \C \cdot  U_{(0,-1)},
\end{split}
\end{equation*}
since the differential is trivial for the degree reason.
\item $\Mo^\mult_\cE(P)(\bL_{T_\cp},L(2))$:\par
The projections of the intersection $\pi'(\bL_{T_\cp}\cap L(2))$ and gradient lines are depicted in Figure \ref{fig:TtoO(2)}.
\begin{figure}[h]
		\centering
\begin{tikzpicture}[scale=1.3]
\draw[dashed](-4,0)  --  (-1,0) -- (-4,3) --cycle;
\draw[dashed](4,0)  --  (1,0) -- (1,3)  --cycle;
\draw[dashed,orange] (-4,1.75) to[out=350,in=120] (-3,1);
\draw[dashed,orange] (1,1.75) to[out=340,in=120] (2,1);
\fill[black](-4,0)circle(0.06) node[left]{$v^0_{(1,0)}$};
\fill[black](-1,0)circle(0.06) node[right]{$v^1_{(1,0)}$};
\fill[black](-4,3)circle(0.06) node[left]{$v^2_{(0,1)}$};
\fill[black](1,0)circle(0.06) node[left]{$v^0_{(0,1)}$};
\fill[black](4,0)circle(0.06) node[right]{$v^1_{(1,1)}$};
\fill[black](1,3)circle(0.06) node[left]{$v^2_{(1,1)}$};
\fill[black](-2.5,0)circle(0.06) node[above]{$V_{(1,0)}$};
\fill[black](2.5,1.5)circle(0.06) node[above right]{$V_{(1,1)}$};
\fill[black](1,1.5)circle(0.06) node[left]{$V_{(0,1)}$};
\draw[thick](1,0) -- (1,1.5);
\draw[arrows = {-Stealth[scale=1.5]}](1,1.5) -- (1,0.75);
\draw[thick](1,1.75) -- (1,1.5);
\draw[arrows = {-Stealth[scale=1.5]}](1,1.5) -- (1,1.75);
\draw[thick](-4,3) -- (-4,1.75);
\draw[arrows = {-Stealth[scale=1.5]}](-4,2) -- (-4,2.5);
\draw[thick](1,3) -- (2.5,1.5);
\draw[arrows = {-Stealth[scale=1.5]}](2.5,1.5) -- (1.75,2.25);
\draw[thick](4,0) -- (2.5,1.5);
\draw[arrows = {-Stealth[scale=1.5]}](2.5,1.5) -- (3.25,0.75);
\draw[thick](-1,0) -- (-2.5,0);
\draw[arrows = {-Stealth[scale=1.5]}](-2.5,0) -- (-1.75,0);
\draw[thick](-4,0) -- (-2.5,0);
\draw[arrows = {-Stealth[scale=1.5]}](-2.5,0) -- (-3.25,0);
\draw (-2.5,-0.6) node{$\pi'(\bL_{T_\cp}^{(1)}\cap L(2))$};
\draw (2.5,-0.6) node{$\pi'(\bL_{T_\cp}^{(2)}\cap L(2))$};
\end{tikzpicture}
	    \caption{The projection of the intersection $\bL_{T_\cp}\cap L(2)$ to the moment polytope and gradient lines.}
		\label{fig:TtoO(2)}
\end{figure}
By the behavior of $\bL_{T_\cp}$ and $L(2)$ over each edge of the moment polytope, each $V_{(j_1,j_2)}$ forms a generator whose degree turns out to be $|V_{(j_1,j_2)}|=0$. On the other hand, the other intersections do not form generators since each $v^k_{(j_1,j_2)}$ does not satisfy Condition {\bf(M)}.
Hence we have
\begin{equation*}
\begin{split}
H^0(\Mo^\mult_\cE(P))(\bL_{T_\cp},L(2)) &\cong \Mo^\mult_\cE(P)(\bL_{T_\cp},L(2))\\
&\cong \C \cdot  V_{(1,1)} \oplus \C \cdot  V_{(1,0)} \oplus \C \cdot  V_{(0,1)},
\end{split}
\end{equation*}
since the differential is trivial for the degree reason.
\item $\Mo^\mult_\cE(P)(L(1),L(2))$:\par
This space is computed in \cite{FK21}.
The projections of the intersection $W_{(i_1,i_2)}\in\pi'(L(1)\cap L(2))$ are depicted in Figure \ref{fig:0toO(1)}.
\begin{figure}[h]
		\centering
\begin{tikzpicture}[scale=1]
\draw[dashed](3,0)  --  (0,0) -- (0,3)  --cycle;
\fill[black](0,0)circle(0.06);
\fill[black](3,0)circle(0.06);
\fill[black](0,3)circle(0.06);
\draw(0,0) node[left]{$W_{(0,0)}$};
\draw(3,0) node[right]{$W_{(1,0)}$};
\draw(0,3) node[left]{$W_{(0,1)}$};
\end{tikzpicture}
	    \caption{The projection of the intersection $L(1) \cap L(2)$ to the moment polytope.}
		\label{fig:0toO(1)}
\end{figure}
These points form the generators whose degrees are zero. We also have
\begin{equation*}
\begin{split}
H^0(\Mo^\mult_\cE(P))(L(1),L(2)) &\cong \Mo^\mult_\cE(P)(L(1),L(2))\\
&\cong \C \cdot  W_{(0,0)} \oplus \C \cdot  W_{(1,0)} \oplus \C \cdot  W_{(0,1)},
\end{split}
\end{equation*}
since the differential is trivial for the degree reason.
\item $\Mo^\mult_\cE(P)(\bL_{T_\cp},\bL_{T_\cp})$:\par
The Lagrangian multi-section $\bL_{T_\cp}$ has self-intersections only at each vertex of the moment polytope since $\bL_{T_\cp}$ can be realized by the embedded Lagrangian in the total space $Y'$.
Furthermore, considering the stable manifolds of each self-intersection, we see that they do not satisfy Condition {\bf(M)}. For instance, at the vertex $\check{\sigma_0}=(0,0)$, the first sheet $\bL_{T_\cp}^{(1)}$ and the second sheet $\bL_{T_\cp}^{(2)}$ of the Lagrangian multi-section $\bL_{T_\cp}=\bL_{T_\cp}^{(1)}\cup\bL_{T_\cp}^{(2)}$ is asymptotically expressed by
\begin{equation}\label{FS chern expr}
\begin{split}
\bL_{T_\cp}^{(1)} = \mathrm{graph}\left(df_{\bL_{T_\cp}}^{(1)}\right):\ \ \ \begin{pmatrix}y^1\\y^2\end{pmatrix} &= \begin{pmatrix}2x^1-1\\x^2\end{pmatrix},\\
\bL_{T_\cp}^{(2)} = \mathrm{graph}\left(df_{\bL_{T_\cp}}^{(2)}\right):\ \ \ \begin{pmatrix}y^1\\y^2\end{pmatrix} &= \begin{pmatrix}x^1\\2x^2-1\end{pmatrix},
\end{split}
\end{equation}
where $f_{\bL_{T_\cp}}^{(1)}$ (resp. $f_{\bL_{T_\cp}}^{(2)}$) is the smoothing of $\varphi^{\trop}_{T_\cp}$. The gradient vector fields are given by
\begin{equation*}
\begin{split}
\grad\left(f_{\bL_{T_\cp}}^{(2)}-(f_{\bL_{T_\cp}}^{(1)}+\xi_1-\xi_2)\right) &= -x^1\frac{\del}{\del x^1} + x^2\frac{\del}{\del x^2},\\
\grad\left(f_{\bL_{T_\cp}}^{(1)}-(f_{\bL_{T_\cp}}^{(2)}-\xi_1+\xi_2)\right) &= x^1\frac{\del}{\del x^1} - x^2\frac{\del}{\del x^2}.
\end{split}
\end{equation*}
The stable manifolds of these gradient vector fields turn out to be the $x^1$-axis or $x^2$-axis, which both cases do not satisfy Condition {\bf(M)} and this self-intersection does not form a generator of the space $\Mo^\mult_\cE(P)(\bL_{T_\cp},\bL_{T_\cp})$. The same is true for the other vertices. Therefore, no self-intersection can be the generator of the space $\Mo^\mult_\cE(P)(\bL_{T_\cp},\bL_{T_\cp})$ and we have
\begin{equation*}
\Mo^\mult_\cE(P)(\bL_{T_\cp},\bL_{T_\cp}) = \C \cdot P^{(1)} \oplus \C \cdot P^{(2)} \oplus \C \cdot b^{(12)}.
\end{equation*}
In addition, there are trivial gradient lines from $P^{(1)}$ and $P^{(2)}$ to $b_{12}$. Hence, we have $\fm_1(P^{(i)})=\pm b_{12}$. Choosing a suitable orientation of the moduli space of gradient lines, we have
\begin{equation*}
H^0\left(\Mo^\mult_\cE(P)\right)(\bL_{T_\cp},\bL_{T_\cp}) = \C \cdot \left(P^{(1)}+P^{(2)}\right).
\end{equation*}
\end{itemize}
\begin{remark}
The expression (\ref{FS chern expr}) of Lagrangian multi-section $\bL_{T_\cp}$ at a vertex is related to the Chern connection $\nabla_{FS}$ of $T_\cp$ associated to the Fubini-Study metric. For the holomorphic frame $\{\frac{\del}{\del u},\frac{\del}{\del v}\}$,  the Chern connection $\nabla_{FS}$ is given by
\begin{equation*}
\nabla_{FS} = d - A_{FS} = d-\frac{1}{1+u\bar{u}+v\bar{v}}\left(\begin{pmatrix} 2\bar{u}&\bar{v}\\0&\bar{u}  \end{pmatrix}du+ \begin{pmatrix} \bar{v}&0\\\bar{u}&2\bar{v}  \end{pmatrix}dv \right).
\end{equation*}
By using $u=e^{z_1}=e^{\xi_1+\ii y_1}$ and $v=e^{z_2}=e^{\xi_2+\ii y_2}$, it can be written as
\begin{equation*}
\begin{split}
A_{FS} &= \dfrac{1}{1+e^{2\xi_1}+e^{2\xi_2}}\begin{pmatrix}
(e^{2\xi_1}-1-e^{2xi_2})dz_1+e^{2\xi_2}dz_2&e^{2\xi_2}dz_1\\
e^{2\xi_1}dz_2&e^{2\xi_1}dz_1+(e^{2\xi_2}-1-e^{2\xi_1})dz_2
\end{pmatrix}\\
&= \begin{pmatrix}
(2x^1-1)dz_1 + x^2dz_2 & x^2dz_1\\
x^1dz_2 & x^1dz_1 + (2x^2-1)dz_2
\end{pmatrix}
\end{split}
\end{equation*}
with respect to the non-holomorphic frame $\{\frac{\del}{\del z_1},\frac{\del}{\del z_2}\}$. From the view point of the SYZ transformation, the $dy$-part of the diagonal component of the connection form corresponds to the mirror Lagrangian multi-section. In the case of a diagonal matrix, the $(k,k)$-entry corresponds to the $k$-th sheet of the mirror Lagrangian multi-section (see also \cite{CS19,Sue21}). In other cases, the correspondence is more complicated. We expect that the diagonal component is related to the mirror Lagrangian multi-section and the other components are related to some correction arising from the branch point. 
\end{remark}

To summarize, the morphisms in the cohomology $H(\Mo^{\mult}_\cE(P))$ are as follows:
\begin{equation*}
\xymatrix{
L(1) \ar@(dl, dr)_{P} \ar@/^18pt/[rr]^{U_{(-1,0)}}  \ar[rr]^{U_{(0,-1)}} \ar@/_18pt/[rr]^{U_{(0,0)}} & &  \bL_{T_\cp} \ar@(dl, dr)_{P^{(1)}+P^{(2)}} \ar@/^18pt/[rr]^{V_{(0,1)}}  \ar[rr]^{V_{(1,0)}} \ar@/_18pt/[rr]^{V_{(1,1)}} & &   L(2) \ar@(dl, dr)_{P}
}
\end{equation*}
Note that there are no morphisms in the reverse direction. This follows from Condition {\bf (M)} and the fact that all critical points lie on $\del P$.\par
Next, we compute the composition $\fm_2$ of morphisms in $\Mo^{\mult}_\cE(P)$. Firstly, we consider the gradient trees ending at $W_{(0,0)}$. There are two ordinary gradient trees:
\begin{align*}
\fm_2(U_{(-1,0)},V_{(1,0)}) &= \pm e^{-A(\gamma_1)}W_{(0,0)},& \fm_2(U_{(0,-1)},V_{(0,1)}) &= \pm e^{-A(\gamma_2)}W_{(0,0)},
\end{align*}
since the composition $\fm_2$ preserves the label $(i_1,i_2)$ and the degree.  The weights of these gradient trees are finite and equal due to the symmetry of the perturbation. We denote by $A:=A(\gamma_1)=A(\gamma_2)$ the weight of the gradient tree.
In contrast, there exist no jagged gradient trees. This can be seen as follows. Firstly, the candidates for the output of $\fm_2$ are determined by the fact that $\fm_2$ preserves the $\mathbb{Z}^2$-grading related to the lift of Lagrangians. Next, by chasing the interactions between one of the gradient lines and the spectral network, the branches of the potential function are not compatible with those of the other gradient line (see Figure \ref{fig:no jagged gradient tree}). This implies the non-existence of jagged gradient trees.
\begin{figure}[h]
		\centering
\begin{tikzpicture}[
    scale=0.7,
    > = {Stealth[length=3mm, width=2mm]},
    midarrow/.style={postaction={decorate}, decoration={markings, mark=at position #1 with {\arrow{>}}}},
    revarrow/.style={postaction={decorate}, decoration={markings, mark=at position #1 with {\arrow{<}}}}
]
\draw[dashed] (0,0) -- (10,0) -- (0,10) -- cycle;
\coordinate (O) at (0,0);
\coordinate (Y1) at (0, 4.5);
\coordinate (X1) at (4.8, 0);
\coordinate (C) at (5.2, 3.2); 
\coordinate (S) at (3.3, 3.3); 
\draw[blue, very thick, midarrow=0.33, midarrow=0.66] (X1) -- (O);
\draw[black, semithick, revarrow=0.5]
    (0, 7.2) node[left]{$(21)$-wall} .. controls (1.5, 6.8) and (4, 5) .. (S);
\draw[black, semithick, midarrow=0.5]
    (S) .. controls (2, 2) and (3, 0) .. (3, 0) node[below]{$(21)$-wall};
\draw[black, semithick]
    (S) .. controls  (4, 2.8) and (4.6, 2.9) .. (C);
\draw[black, semithick, midarrow=0.6]
    (C) .. controls  (5.8, 3.5) and (6, 3.7) .. (6.15, 3.85) node[above right]{$(12)$-wall};
\draw[orange, dashed]
    (S) .. controls (2.0, 4.6) and (1.0, 5.2) .. (0, 5.6) node[left]{$(12)$-branch cut};
\draw[red, very thick, midarrow=0.3, midarrow=0.8]
    (Y1) .. controls (2.5, 4.5) and (4.2, 4.0) .. (C);
\draw[red, very thick, midarrow=0.5]
    (C) .. controls (4.3, 2.4) and (3.5, 2) .. (2.78,2.5);
\draw[red, very thick, midarrow=0.5]
    (2.78,2.5) .. controls (2, 2.4) and (1, 1.4) .. (O);
\fill[black] (O) circle (2.5pt) node[below left]{$W_{(0,0)}$};
\fill[black] (Y1) circle (2.5pt) node[left]{$U_{(-1,0)}$};
\fill[black] (X1) circle (2.5pt) node[below]{$V_{(1,0)}$};
\fill[red] (C) circle (1.5pt);
\fill[red] (2.78,2.5) circle (1.5pt);
\draw (1.5,4) node[red]{$(1)$};
\draw (2.7,4.6) node[red]{$(2)$};
\draw (4.5,4.2) node[red]{$(2)$};
\draw (4.5,2) node[red]{$(1)$};
\draw (1.5,1) node[red]{$(2)$};
\end{tikzpicture}
	    \caption{The jagged gradient line emanating from $U_{(-1,0)}$. The flow of $\grad(f^{(1)}_{\bL_{T_\cp}}-f_{L(1)})$ is represented by $(1)$ and the flow $\grad(f^{(2)}_{\bL_{T_\cp}}-f_{L(1)})$ is represented by $(2)$.
If that line reaches $W_{(0,0)}$, then it firstly starts by $\grad(f^{(1)}_{\bL_{T_\cp}}-f_{L(1)})$ and finally proceeds by $\grad(f^{(2)}_{\bL_{T_\cp}}-f_{L(1)})$. However the gradient line emanating from $V_{(-1,0)}$ proceeds by $\grad(f_{L(2)}-f^{(1)}_{\bL_{T_\cp}})$. Since the branches of $f_{\bL_{T_\cp}}$ do not coincide, these lines do not form a jagged gradient tree.}
		\label{fig:no jagged gradient tree}
\end{figure}
The gradient trees ending at $W_{(1,0)}$ and $W_{(0,1)}$ are the same as above, namely
\begin{align*}
\fm_2(U_{(0,0)},V_{(1,0)}) &= \pm e^{-B}W_{(1,0)},& \fm_2(U_{(0,-1)},V_{(1,1)}) &= \pm e^{-B}W_{(1,0)},\\
\fm_2(U_{(0,0)},V_{(0,1)}) &= \pm e^{-C}W_{(0,1)},& \fm_2(U_{(-1,0)},V_{(1,1)}) &= \pm e^{-C}W_{(0,1)}.
\end{align*}
It remains to determine the sign $\pm$ for each product. These signs are given by the orientation of the moduli space $\cM(U_{(i_1,i_2)},V_{(j_1,j_2)};W_{(i_1+j_1,i_2+j_2)})$ of gradient trees. Also, the orientation of $\cM(U_{(i_1,i_2)},V_{(j_1,j_2)};W_{(i_1+j_1,i_2+j_2)})$ is induced by the orientations of the unstable manifold associated to $U_{(i_1,i_2)},V_{(j_1,j_2)}$ and of the stable manifold associated to $W_{(i_1+j_1,i_2+j_2)}$.
By choosing a suitable orientation for each unstable manifold, we obtain the following:
\begin{align}\label{rel in Mo(P)}
\fm_2(U_{(-1,0)},V_{(1,0)}) &= - e^{-A}W_{(0,0)},& \fm_2(U_{(0,-1)},V_{(0,1)}) &= e^{-A}W_{(0,0)},\notag\\
\fm_2(U_{(0,0)},V_{(1,0)}) &= - e^{-B}W_{(1,0)},& \fm_2(U_{(0,-1)},V_{(1,1)}) &= e^{-B}W_{(1,0)},\\
\fm_2(U_{(0,0)},V_{(0,1)}) &= e^{-C}W_{(0,1)},& \fm_2(U_{(-1,0)},V_{(1,1)}) &= -e^{-C}W_{(0,1)}.\notag
\end{align}
These compositions induce the compositions in the cohomology $H(\Mo^\mult_\cE(P))$.
Finally, we consider the compositions with endomorphisms of $\bL_{T_\cp}$. The morphism $P^{(r)}$ is regarded as the projection $\pi'(\bL^{(r)}_{T_\cp}\cap\bL^{(r)}_{T_\cp})$. Hence, if $U_{(i_1,i_2)}\in\pi'(L(1)\cap\bL_{T_\cp}^{(r)})$, then $\fm_2(U_{(i_1,i_2)},P^{(r)})=U_{(i_1,i_2)}$ since there is the only trivial gradient tree. It is the same as the morphisms $V_{(j_1,j_2)}$. Therefore, $P^{(1)}+P^{(2)}$ forms the cohomological unit, i.e., $P^{(1)}+P^{(2)}$ is the identity morphism in the cohomology $H(\Mo^\mult_\cE(P))(\bL_{T_\cp},\bL_{T_\cp})$. For the degree reason, the cohomology $H(\Mo_\cE^\mult(P))$ has no higher products. In contrast, $Mo_\cE^\mult(P)$ has higher products such as $\fm_3(U_{(i_1,i_2)},b^{(12)},V_{(j_1,j_2)})$. From the $A_\infty$-relation corresponding to the associativity, there must exist nontrivial higher products. However, in this case, we can define a non-trivial $A_\infty$-structure inductively and algebraically\footnote{The reason we say ``algebraically" here is that since self-hom is defined somewhat algebraically, we also define its product algebraically.} from lower orders. Since the higher products of $H(\Mo_\cE^\mult(P))$ vanish, all $A_\infty$-structures of $\Mo_\cE^\mult(P))$ are quasi-isomorphic. Therefore, the cohomology $H(\Mo_\cE^\mult(P))$ is independent of the higher products of $\Mo_\cE^\mult(P)$.

\subsection{Proof of the main theorem}\label{proof of main thm}
In this subsection, we prove the main theorem. We consider the dg-category to be an $A_\infty$-category in which higher products $\fm_k\ (k\geq3)$ vanish. We firstly prove the following.
\begin{lemma}\label{main lem}
We have an $A_\infty$-equivalence\footnote{An \textit{$A_\infty$-equivalence} is a cohomologically unital $A_\infty$-functor $\mathcal{F}: \mathcal{C} \to \mathcal{D}$ between two cohomologically unital $A_\infty$-categories such that the induced functor $H(\mathcal{F}): H(\mathcal{C}) \to H(\mathcal{D})$ is an equivalence (see \cite{Sei08}, where it is called a \textit{quasi-equivalence}).}
\begin{equation*}
\iota: H(\Mo^{\mult}_\cE(P)) \simeq H(\DG^\vect_\cE(\cp)).
\end{equation*}
\end{lemma}
\begin{proof}
Firstly, the higher products of $H(\Mo^{\mult}_\cE(P))$ vanish since morphisms in $H(\Mo^{\mult}_\cE(P))$ are of degree zero. Namely, we can regard $H(\Mo^{\mult}_\cE(P))$ as a dg category in the usual sense.
The correspondence between the objects is discussed in subsection \ref{correspondence between objects}. For both categories $H(\DG^\vect_\cE(\cp))$ and $H(\Mo^{\mult}_\cE(P))$, we computed the compositions of morphisms in section \ref{complex side} and subsection \ref{morphisms in mult Mo(P)}. 
We assign the correspondence of morphisms as follows: 
\begin{align*}
\iota(U_{(-1,0)}) &= F_1, & \iota(U_{(0,-1)}) &= F_2, & \iota(U_{(0,0)}) &= F_3,\\
\iota(V_{(0,1)}) &= G_1, & \iota(V_{(1,0)}) &= G_2, & \iota(V_{(1,1)}) &= G_3\\
\iota(W_{(0,0)}) &= e^{A} \cdot 1, & \iota(W_{(1,0)}) &= e^{B} \cdot u, & \iota(W_{(0,1)}) &= e^{C} \cdot v.
\end{align*}
This assignment is compatible with the compositions (\ref{rel in DGvect}) and (\ref{rel in Mo(P)}). Namely, we have
\begin{equation*}
\iota(\fm_2(U_{(i_1,i_2)},V_{(j_1,j_2)})) = \iota(V_{(j_1,j_2)}) \circ \iota(U_{(i_1,i_2)}).
\end{equation*}
Thus, this assignment turns out to be an $A_\infty$-equivalence. 
\end{proof}
\begin{theorem}\label{main theorem prov}
We have an $A_\infty$-equivalence
\begin{equation*}
\Mo^{\mult}_\cE(P) \simeq \DG^\vect_\cE(\cp).
\end{equation*}
\end{theorem}
\begin{proof}
For any $A_\infty$-category $(\mathcal{C},\{m_k\}_{k\geq1})$, there exists an $A_\infty$-category $(H(\mathcal{C}),\{m'_k\}_{k\geq1})$ and an $A_\infty$-equivalence $(H(\mathcal{C}),\{m'_k\}_{k\geq1}) \to (\mathcal{C},\{m_k\}_{k\geq1})$. This is known as the minimal model theorem. For more details, see \cite{Sei08, Kaj13}.
We apply the minimal model theorem to $\Mo^{\mult}_\cE(P)$, $\DG^\vect_\cE(\cp)$ and denote by $p: H(\Mo^\mult_\cE(P)) \to \Mo^\mult_\cE(P)$, $q: H(\DG^\vect_\cE(\cp)) \to \DG^\vect_\cE(\cp)$ the $A_\infty$-equivalences.
Also, any $A_\infty$-equivalence has its inverse $A_\infty$-equivalence. Hence, we can consider the composition $q \circ \iota \circ p^{-1}: \Mo^\mult_\cE(P) \to \DG^\vect_\cE(\cp)$ where $\iota: H(\Mo^{\mult}_\cE(P)) \to H(\DG^\vect_\cE(\cp))$ is the $A_\infty$-equivalence constructed in Lemma \ref{main lem}.
The composition of $A_\infty$-equivalences is an $A_\infty$-equivalence. Therefore, $q \circ \iota \circ p^{-1}$ is an $A_\infty$-equivalence.
\begin{equation*}
\xymatrix{
\Mo^\mult_\cE(P) \ar@{-->}[r]^{q\circ \iota \circ p^{-1}}& \DG^\vect_\cE(\cp)\\
H(\Mo^\mult_\cE(P)) \ar[u]^p \ar[r]_{\iota}& H(\DG^\vect_\cE(\cp))\ar[u]_q
}
\end{equation*}
\end{proof}
\begin{corollary}\label{main cor prov}
We have an equivalence of triangulated categories
\begin{equation*}
Tr(\Mo^{\mult}_\cE(P)) \simeq D^b(\Coh(\cp)).
\end{equation*}
\end{corollary}
\begin{proof}
Since $\cE$ is a full (strongly) exceptional collection of $D^b(\Coh(\cp))$, we have an equivalence of triangulated categories
\begin{equation}\label{B-side equiv}
Tr(\DG^\vect_\cE(\cp))  \simeq D^b(\Coh(\cp)).
\end{equation}
For two $A_\infty$-equivalent $A_\infty$-categories $\mathcal{C}$ and $\mathcal{D}$, the induced triangulated categories $Tr(\mathcal{C})$ and $Tr(\mathcal{D})$ are equivalent as triangulated categories. Hence, by Theorem \ref{main theorem prov}, we have
\begin{equation}\label{main equiv}
Tr(\Mo^\mult_\cE(P)) \simeq Tr(\DG^\vect_\cE(\cp)).
\end{equation}
Combining the equivalences (\ref{B-side equiv}) and (\ref{main equiv}), the proof is complete.
\end{proof}
\begin{remark}
The products \eqref{rel in Mo(P)} are symmetric since we take the perturbation of $\bL_{T_\cp}$ in the neighborhood of $\del P$ symmetrically. We can construct an $A_\infty$-equivalence between $H(\Mo^\mult_\cE(P))$ and $H(\DG^\vect_\cE(\cp))$ even if the perturbation is not symmetric.
Since the Lagrangian is perturbed by the Hamiltonian diffeomorphism, these weights are represented by the values of the Hamiltonian at the vertices of the moment polytope by Stokes' theorem.
This allows us to construct an $A_\infty$-equivalence by appropriately assigning coefficients to each assignment $U \mapsto F$ and $V \mapsto G$. 
This shows that the cohomology $H(\Mo^\mult_\cE(P))$ is independent of the perturbations by the Hamiltonian diffeomorphisms.
\end{remark}

\begin{remark}\label{necessity of jagged lines}
Fortunately, the proof of the main theorem do not require jagged gradient lines or jagged gradient trees. However, in general, these are necessary. In Section \ref{Example of nontrivial jagged gradient tree in CP2}, we consider a non-trivial jagged gradient tree and the composition associated to the tree in $\Mo^\mult(P)$. We will see that this non-trivial composition corresponds to the composition of morphisms between toric vector bundles in $\DG^\vect(\cp)$.
\end{remark}

\section{The global sections of the tangent bundle}\label{global sections and its mirror}
In this section, we consider the space of morphisms from the structure sheaf $\cO$ to the holomorphic tangent bundle $T_\cp$ and the holomorphic cotangent bundle $\Omega_{\cp}$.
Naively considering the mirror side of this, the two mirror Lagrangians intersect at the ramification locus.
Since the definition of $Mo^{\mult}(P)$ excludes such cases, we need to perturb the Lagrangians. However, by allowing such cases and formally adding a morphism corresponding to the intersection at the ramification locus, we can easily compute the space of morphisms.

\subsection{The morphisms from the structure sheaf to the tangent bundle}
By the Euler sequence
\begin{equation*}
0 \longrightarrow \cO \longrightarrow \cO(1)^{\oplus3} \longrightarrow T_{\cp} \longrightarrow 0,
\end{equation*}
we have
\begin{equation*}
\begin{split}
H^k(\DG^\vect(\cp))(\cO,T_{\cp}) &\cong \begin{cases}\Gamma(\cp,T_{\cp}) & (k=0)\\ 0 & (\text{otherwise}) \end{cases}\\
\dim H^0(\DG^\vect(\cp))(\cO,T_{\cp}) &= \dim \Gamma(\cp,T_{\cp}) = 8,
\end{split}
\end{equation*}
where $\Gamma(\cp,T_{\cp})$ is the space of global sections.
In order to express the global sections, we use the notion of the parliament of polytopes associated to toric vector bundles introduced by Di Rocco-Jabbusch-Smith in \cite{DJS14}.
For a toric vector bundle, the lattice points in the parliament of polytopes  correspond to the generators of the space of global sections.
This expression is a generalization of the fact that the global sections of line bundles over a toric variety can be expressed by using a polytope and lattice points. For more details see \cite{CLS11,DJS14}. Figure \ref{fig:PPTCP2} shows the parliament of polytopes associated to the holomorphic tangent bundle $T_\cp$.
\begin{figure}[h]
\centering
\begin{tikzpicture}
\draw(-2,0) --  (2,0);
\draw(0,-2) --  (0,2);
\draw(-2,2) --  (2,-2);
\draw(1.5,0) -- (0,0) -- (0,1.5) --cycle; 
\draw(-1.5,1.5) --  (-1.5,0) -- (0,0) --cycle; 
\draw(0,-1.5) --  (0,0) -- (1.5,-1.5) --cycle; 
\fill[red,opacity=0.2](1.5,0) -- (0,0) -- (0,1.5) --cycle;
\fill[cyan,opacity=0.2](-1.5,0) --  (0,0) -- (-1.5,1.5) --cycle; 
\fill[green,opacity=0.2](0,-1.5) --  (0,0) -- (1.5,-1.5) --cycle;
\fill[black](0,0)circle(0.06);
\fill[black](0,1.5)circle(0.06);
\fill[black](1.5,0)circle(0.06);
\fill[black](-1.5,0)circle(0.06);
\fill[black](0,-1.5)circle(0.06);
\fill[black](1.5,-1.5)circle(0.06);
\fill[black](-1.5,1.5)circle(0.06);
\draw[cyan](-2.1,0.75)node{$P_{(-1,0)}$};
\draw[red](1,1)node{$P_{(1,1)}$};
\draw[green](1.5,-0.75)node{$P_{(0,-1)}$};
\draw (1.5,0) node[below]{$1$};
\draw (-1.5,0) node[below]{$-1$};
\draw (0,1.5) node[left]{$1$};
\draw (0,-1.5) node[left]{$-1$};
\end{tikzpicture}
\caption{The parliament of polytopes associated to $T_\cp$.}
\label{fig:PPTCP2}
\end{figure}
The lattice points in these polytopes correspond to the global sections as follows:
\begin{align}\label{globalsectionPP}
(-1,0)&\otimes\chi^{-(-1,1)}, &  & &  (1,1)&\otimes\chi^{-(0,1)}, &  & \notag\\
(-1,0)&\otimes\chi^{-(-1,0)}, & (-1,0)&\otimes\chi^{-(0,0)},  & (1,1)&\otimes\chi^{-(0,0)}, & (1,1)&\otimes\chi^{-(1,0)}, \notag\\
& & & & (0,-1)&\otimes\chi^{-(0,0)}, & & \\
& & & & (0,-1)&\otimes\chi^{-(0,-1)}, & (0,-1)&\otimes\chi^{-(1,-1)}.\notag
\end{align}
By this expression, we have $\dim \Gamma(\cp,T_{\cp}) = 8$ since $(0,-1)\otimes\chi^{(0,0)} + (1,1)\otimes\chi^{(0,0)} + (-1,0)\otimes\chi^{(0,0)} =0$.\par
On the other hand, we compute the morphisms in the mirror side.
The Lagrangian multi-section $\bL_{T_\cp}$ associated to the holomorphic tangent bundle $T_\cp$ is given in subsection \ref{correspondence between objects}.
Figure \ref{fig:asympt cond} shows the projection $V_{(i_1,i_2)}$ of the intersection between the zero section $L(0)$ and the Lagrangian multi-section $\bL_{T_\cp}$ to the moment polytope. 
This follows from the behavior near the boundary of the moment polytope and the connectedness of the Lagrangian.
\begin{figure}[h]
		\centering
\begin{tikzpicture}[scale=1.3]
\draw[dashed](-4,0)  --  (-1,0) -- (-4,3) --cycle;
\draw[dashed](5,0)  --  (2,0) -- (2,3)  --cycle;
\draw[dashed,orange] (-4,1.25) to[out=10,in=160] (-3,1);
\draw[dashed,orange] (2,1.25) to[out=10,in=160] (3,1);
\fill[black](-4,0)circle(0.06);
\fill[black](-1,0)circle(0.06);
\fill[black](2,0)circle(0.06);
\fill[black](5,0)circle(0.06);
\fill[black](-4,3)circle(0.06);
\fill[black](2,3)circle(0.06);
\fill[black](-2.5,0)circle(0.06);
\fill[black](-4,1.5)circle(0.06);
\fill[black](3.5,1.5)circle(0.06);
\fill[black](-3,1)circle(0.06);
\fill[black](3,1)circle(0.06);
\draw(-4,0) node[left]{$V_{(-1,0)}$};
\draw (2,0) node[left]{$V_{(0,-1)}$};
\draw(5,0) node[right]{$V_{(1,-1)}$};
\draw(-1,0) node[right]{$V_{(1,0)}$};
\draw (-4,3) node[left]{$V_{(0,1)}$};
\draw (2,3) node[left]{$V_{(-1,1)}$};
\draw(-2.5,0) node[above]{$V_{(0,0)}$};
\draw(-4,1.5) node[left]{$V_{(0,0)}$};
\draw(3.5,1.5) node[above right]{$V_{(0,0)}$};
\draw(-3,1) node[right]{$V_{(0,0)}^b$};
\draw(3,1) node[right]{$V_{(0,0)}^b$};
\draw (-2.5,-0.6) node{$\pi'(L(0)\cap\bL_{T_\cp}^{(1)})$};
\draw (3.5,-0.6) node{$\pi'(L(0)\cap\bL_{T_\cp}^{(2)})$};
\end{tikzpicture}
	    \caption{The projection of the intersection $L(0)\cap\bL_{T_\cp}$ to the moment polytope.}
		\label{fig:asympt cond}
\end{figure}
The degree of $V_{(i_1,i_2)}$ is equal to $0$ except for $V^b_{(0,0)}$. 
Indeed, near the boundary $\del P$ of the moment polytope, each sheet of the Lagrangian multi-section $\bL_{T_\cp}$ is a Lagrangian section and behaves like $L(1)$ or $L(2)$ asymptotically. Hence the degree of $V_{(i_1,i_2)}$ is equal to $0$.
In contrast, we can not define the degree of $V^b_{(0,0)}$ since the potential function of $\bL_{T_{\cp}}$ around $V^b_{(0,0)}$ is not a Morse function. However, we can compute the space $\Mo^{\mult}(P)(L(0),\bL_{T_\cp})$ of morphisms by formally adding $V^b_{(0,0)}$ as a generator to $\Mo^{\mult}(P)(L(0),\bL_{T_\cp})$ and defining the degree\footnote{This formal degree is related to the Maslov index of the intersection point of two Lagrangians. In this case, the Maslov index of the intersection $L(0)\cap\bL_{T_\cp}$ is equal to $1$.} of $V^b_{(0,0)}$ to be $1$. 
Figure \ref{fig:grad traj of tangent bundle} shows the gradient lines associated to $V_{(0,0)}=V^{0}_{(0,0)} \cup V^{1}_{(0,0)} \cup V^{2}_{(0,0)}$ and $V^b_{(0,0)}$. 
\begin{figure}[h]
		\centering
\begin{tikzpicture}[scale=1.5]
\draw[dashed](-4,0)  --  (-1,0) -- (-4,3) --cycle;
\draw[dashed](4,0)  --  (1,0) -- (1,3)  --cycle;
\draw[dashed,orange] (-4,1.25) to[out=10,in=160] (-3,1);
\draw[dashed,orange] (1,1.25) to[out=10,in=160] (2,1);
\fill[black](-2.5,0)circle(0.06);
\fill[black](-4,1.5)circle(0.06);
\fill[black](2.5,1.5)circle(0.06);
\fill[black](-3,1)circle(0.06);
\fill[black](2,1)circle(0.06);
\draw(-2.5,0) node[above right]{$V^2_{(0,0)}$};
\draw(-4,1.5) node[left]{$V^1_{(0,0)}$};
\draw(2.5,1.5) node[above right]{$V^0_{(0,0)}$};
\draw(-3,1) node[above right]{$V_{(0,0)}^b$};
\draw(2,1) node[below]{$V_{(0,0)}^b$};
\draw[-stealth,postaction=decorate,
    decoration={markings,
    mark = between positions 0.2 and 1 step 0.3
    with {\arrow{stealth}}}](-2.5,0)to[out=140,in=240](-3,1);
\draw[-stealth,postaction=decorate,
    decoration={markings,
    mark = between positions 0.2 and 1 step 0.3
    with {\arrow{stealth}}}](-4,1.5)to[out=10,in=120](-3,1);
\draw[-stealth,postaction=decorate,
    decoration={markings,
    mark = between positions 0.2 and 1 step 0.3
    with {\arrow{stealth}}}](2.5,1.5)to[out=280,in=340](2,1);
\end{tikzpicture}
	    \caption{The gradient lines emanating from $V_{(0,0)}$ and ending at $V^b_{(0,0)}$.}
		\label{fig:grad traj of tangent bundle}
\end{figure}
Then, we have
\begin{align*}
\fm_1(V^0_{(0,0)}) &= e^{-A}V^b_{(0,0)}, & \fm_1(V^1_{(0,0)}) &= e^{-B}V^b_{(0,0)}, & \fm_1(V^2_{(0,0)}) &= e^{-C}V^b_{(0,0)},
\end{align*}
where $A,B,C\in\R$ are the weights associated to the gradient lines. We also have $\fm_1(V^b_{(0,0)})=0$ by the degree reason.
To summarize, we obtain the following results.
\begin{theorem}
The zero-th cohomology $H^0(\Mo^{\mult}(P))(L(0),\bL_{T_{\cp}})$ is written as
\begin{equation*}
\bigoplus_{(i,j)\in PP(T_\cp)\backslash\{(0,0)\}}\C \cdot V_{(i,j)} \oplus \C \cdot (e^BV^1_{(0,0)}-e^AV^0_{(0,0)}) \oplus \C \cdot (e^CV^2_{(0,0)}-e^AV^0_{(0,0)}) 
\end{equation*}
where $PP(T_\cp)$ is the set of lattice points in the polytopes of the parliament depicted in Figure \ref{fig:PPTCP2}. Furthermore, the higher cohomology $H^k(\Mo^{\mult}(P))(L(0),\bL_{T_{\cp}})\ (k\neq0)$ vanishes.
\end{theorem}
\begin{corollary}\label{cohomology of tangent}
We have an isomorphism
\begin{equation*}
H^k(\Mo^{\mult}(P))(L(0),\bL_{T_{\cp}}) \cong H^k(\DG^\vect(\cp))(\cO,T_{\cp}).
\end{equation*}
\end{corollary}
\begin{proof}
The correspondence between the generators of these cohomologies is given as follows:
\begin{align}\label{Vtochi}
V_{(-1,0)} &\leftrightarrow (-1,0)\otimes\chi^{-(-1,0)},& V_{(0,-1)} &\leftrightarrow  (0,-1)\otimes\chi^{-(0,-1)}, \notag\\
V_{(1,0)} &\leftrightarrow  (1,1)\otimes\chi^{-(1,0)},& V_{(1,-1)} &\leftrightarrow  (0,-1)\otimes\chi^{-(1,-1)},\\\notag
V_{(0,-1)} &\leftrightarrow  (0,-1)\otimes\chi^{-(0,-1)}, & V_{(-1,1)} &\leftrightarrow (-1,0)\otimes\chi^{-(-1,1)},\\
-e^{A}V^0_{(0,0)} &\leftrightarrow  (1,1)\otimes\chi^{-(0,0)}, & e^{B}V^1_{(0,0)} &\leftrightarrow  (-1,0)\otimes\chi^{-(0,0)}, & e^{C}V^2_{(0,0)} &\leftrightarrow  (0,-1)\otimes\chi^{-(0,0)}. \notag
\end{align}
This correspondence turns out to be an isomorphism.
\end{proof}
For any toric vector bundle, we expect that we can get an explicit correspondence between the generators of the space of morphisms such as the correspondence (\ref{Vtochi}).

\subsection{The morphisms from the structure sheaf to the cotangent bundle}
Next, we consider the holomorphic cotangent bundle $\Omega_{\cp}$. Then, we have
\begin{equation*}
\begin{split}
H^k(\DG^\vect(\cp))(\cO,\Omega_{\cp}) &\cong \begin{cases}H^{1,1}(\cp) & (k=1)\\ 0 & (\text{otherwise}) \end{cases}\\
\dim H^{1,1}(\cp) &= 1.
\end{split}
\end{equation*}
In particular, $H^{1,1}(\cp)$ is spanned by the Fubini-Study form $\omega_{FS}$.\par
On the other hand, the tropical Lagrangian multi-section $\bL_{\Omega_\cp}$ associated to the holomorphic cotangent bundle $\Omega_\cp$ is given by Figure \ref{fig:T*cp2}.
\begin{figure}[h]
		\centering
\begin{tikzpicture}[scale=0.8]
\draw(3,0) node[right][black]{$\ \ \ \xrightarrow{\ \ \ p\ \ \ }$}  --  (0,0);
\draw(0,0) --  (-1.5,-2.6) ;
\draw(-1.5,2.6)  --  (0,0);
\draw(0,0) --  (-3,0);
\draw(1.5,2.6)  --  (0,0);
\draw(0,0) --  (1.5,-2.6);
\draw(0.1,0.9) node{$\sigma_0^{(1)}$};
\draw(-0.8,0.4) node{$\sigma_1^{(1)}$};
\draw(-0.8,-0.4) node{$\sigma_2^{(1)}$};
\draw(0,-0.9) node{$\sigma_0^{(2)}$};
\draw(0.8,-0.4) node{$\sigma_1^{(2)}$};
\draw(0.8,0.4) node{$\sigma_2^{(2)}$};
\draw(0,2) node{$\xi_1$};
\draw(-2,1) node{$-\xi_1$};
\draw(-2,-1) node{$-\xi_2$};
\draw(0,-2) node{$\xi_2$};
\draw(2,-1) node{$-\xi_1+\xi_2$};
\draw(2,1) node{$\xi_1-\xi_2$};
\end{tikzpicture}
\begin{tikzpicture}[scale=0.9]
\draw(0,0) --  (-2,0) node[left]{$\rho_1$};
\draw(0,0) --  (0,-2) node[below]{$\rho_2$};
\draw(0,0) --  (1.4,1.4) node[above right]{$\rho_0$};
\draw(-1,-1) node{$\sigma_{0}$};
\draw(1,-1) node{$\sigma_{1}$};
\draw(-1,1) node{$\sigma_{2}$};
\end{tikzpicture}
	    \caption{The tropical Lagrangian multi-section $\bL_{\Omega_{\cp}}^{trop}$ associated to holomorphic cotangent bundle of $\cp$.}
		\label{fig:T*cp2}
\end{figure}
We have $\bL_{\Omega_\cp}=-\bL_{T_\cp}$ by the duality of the tropical Lagrangian multi-sections associated to the toric vector bundles. By taking the potential function $-f_{L_2}$ of the local model instead of $f_{L_2}$, we can construct the Lagrangian multi-section $\bL_{\Omega_\cp}$ such as $\bL_{\Omega_\cp}=-\bL_{T_\cp}$. Hence, the intersections $L(0)\cap\bL_{\Omega_\cp}$ coincide with $L(0)\cap\bL_{T_\cp}$. Figure \ref{fig:asympt cond*} shows the projection of the intersections to the moment polytope.
\begin{figure}[h]
		\centering
\begin{tikzpicture}[scale=1.2]
\draw[dashed](-4,0)  --  (-1,0) -- (-4,3) --cycle;
\draw[dashed](5,0)  --  (2,0) -- (2,3)  --cycle;
\draw[dashed,orange] (-4,1.25) to[out=10,in=160] (-3,1);
\draw[dashed,orange] (2,1.25) to[out=10,in=160] (3,1);
\fill[black](-4,0)circle(0.06);
\fill[black](-1,0)circle(0.06);
\fill[black](2,0)circle(0.06);
\fill[black](5,0)circle(0.06);
\fill[black](-4,3)circle(0.06);
\fill[black](2,3)circle(0.06);
\fill[black](-2.5,0)circle(0.06);
\fill[black](-4,1.5)circle(0.06);
\fill[black](3.5,1.5)circle(0.06);
\fill[black](-3,1)circle(0.06);
\fill[black](3,1)circle(0.06);
\draw(-4,0) node[left]{$U_{(1,0)}$};
\draw (2,0) node[left]{$U_{(0,1)}$};
\draw(5,0) node[right]{$U_{(-1,1)}$};
\draw(-1,0) node[right]{$U_{(-1,0)}$};
\draw (-4,3) node[left]{$U_{(0,-1)}$};
\draw (2,3) node[left]{$U_{(1,-1)}$};
\draw(-2.5,0) node[above]{$U^2_{(0,0)}$};
\draw(-4,1.5) node[left]{$U^1_{(0,0)}$};
\draw(3.5,1.5) node[above right]{$U^0_{(0,0)}$};
\draw(3,1) node[below right]{$U_{(0,0)}^b$};
\draw (-2.5,-0.6) node{$\pi'(L(0)\cap\bL_{\Omega_\cp}^{(1)})$};
\draw (3.5,-0.6) node{$\pi'(L(0)\cap\bL_{\Omega_\cp}^{(2)})$};
\end{tikzpicture}
	    \caption{The projection of the intersection $L(0)\cap\bL_{\Omega_\cp}$ to the moment polytope.}
		\label{fig:asympt cond*}
\end{figure}
Note that $U_{(i_1,i_2)}=V_{(-i_1,-i_2)}$ since $\bL_{\Omega_\cp}=-\bL_{T_\cp}$. Also, the gradient vector field associated to $U_{(i_1,i_2)}$ is the reverse of that associated to $V_{(i_1,i_2)}$. Hence, each point $U_{(i_1,i_2)}$ does not form a generator of $Mo^{\mult}(P)(L(0),\bL_{\Omega_\cp})$ since $U_{(i_1,i_2)}$ does not satisfy Condition {\bf (M)}. The branch point $U^b_{(0,0)}$ is the only generator of the space $\Mo^{\mult}(P)(L(0),\bL_{\Omega_\cp})$ and its degree is equal to $1$. Thus, we obtain the following.
\begin{theorem}\label{cohomology of cotangent}
\begin{equation*}
\begin{split}
H^k(\Mo^{\mult}(P))(L(0),\bL_{\Omega_{\cp}}) &\cong \begin{cases}\C \cdot U^b_{(0,0)} & (k=1)\\0 & (\text{{\rm otherwise}})\end{cases}\\
H^k(\Mo^{\mult}(P))(L(0),\bL_{\Omega_{\cp}}) &\cong H^k(\DG^\vect(\cp))(\cO,\Omega_{\cp}).
\end{split}
\end{equation*}
\end{theorem}

\section{Examples of non-trivial jagged gradient trees}\label{Example of nontrivial jagged gradient tree in CP2}
As mentioned in Remark \ref{necessity of jagged lines}, the jagged gradient lines or trees did not appear in the full subcategory $\Mo^\mult_\cE(P) \subset \Mo^\mult(P)$. However, considering mirror symmetry, jagged gradient lines and trees are necessary in $Mo^\mult(P)$. In this section, we give an example of jagged gradient trees and compare the composition of morphisms associated to the tree with the composition of mirror morphisms.\par
We consider the indecomposable rank $2$ toric vector bundles on $\cp$ by the following exact sequence
\begin{equation*}
0 \longrightarrow \cO \longrightarrow \cO(aD_{0})\oplus\cO(bD_{1})\oplus\cO(cD_{2}) \longrightarrow E_{a,b,c} \longrightarrow 0
\end{equation*}
for $a,b,c \in \Z_{>0}$. The first map is given by $1\mapsto(z_0^a,z_1^b,z_2^c)$ where $[z_0:z_1:z_2]$ are the homogeneous coordinates of $\cp$ and $D_{i} = \{z_i=0\}$. For simplicity, we consider equivariant homomorphisms and compute the following composition
\begin{equation*}
\cO(D_{1}-2D_{2}) \overset{\alpha}\longrightarrow E_{1,2,2} \overset{\beta}\longrightarrow \cO(D_0 + 3D_{1}),
\end{equation*}
where $\alpha:\cO(D_{1}-2D_{2}) \to E_{1,2,2}$ and $\beta:E_{1,2,2} \to \cO(D_0 + 3D_{1})$ are given by
\begin{equation*}
\alpha = \begin{pmatrix} z_1z_2^2 \\ 0 \end{pmatrix},\ \ \ \beta = \begin{pmatrix} z_0z_1 & 0 \end{pmatrix}
\end{equation*}
By this expression, we have the non-trivial composition $\beta \circ \alpha = z_0z_1^2z_2^2 \neq0$.\par
On the other hand, we consider the mirror Lagrangians and morphisms. When we consider the equivariant homological mirror symmetry, we consider the Lagrangians in $T^*\!N_\R$ without taking the quotient by the lattice $M$ (see \cite{FLTZ12,Kuw17}). The mirror Lagrangians $L_{\cO(aD_{0} + bD_{1} + cD_{2})}$ of $\cO(L_{\cO(aD_{0} + bD_{1} + cD_{2})})$ are given by
\begin{equation*}
L_{\cO(aD_{0} + bD_{1} + cD_{2})}:\ 
\begin{pmatrix}
y^1\\y^2
\end{pmatrix}
=
(a+b+c)\begin{pmatrix}
\frac{e^{2\xi_1}}{1+e^{2\xi_1}+e^{2\xi_2}} \\ \frac{e^{2\xi_2}}{1+e^{2\xi_1}+e^{2\xi_2}}
\end{pmatrix} - \begin{pmatrix}
b\\c
\end{pmatrix}
=
(a+b+c)\begin{pmatrix}
x^1\\x^2
\end{pmatrix} - \begin{pmatrix}
b\\c
\end{pmatrix}.
\end{equation*}
The mirror Lagrangians $\bL_{a,b,c}$ of $E_{a,b,c}$ are constructed in \cite{OS24}. The fiber coordinate $(y^1,y^2)$ of $\bL_{1,2,2}$ at each vertex of the moment polytope is depicted in Figure \ref{fig:fib coord of L122}.
\begin{figure}[h]
		\centering
\begin{tikzpicture}[scale=1.3]
\draw(-4,0)  --node[below]{$L(3)$}  (-1,0) -- node[above right]{$L(1)$} (-4,3) -- node[above left]{$L(3)$} node[below left]{$L(2)$}cycle;
\draw(4,0)  -- node[below]{$L(2)$} (1,0) -- node[above left]{$L(2)$} node[below left]{$L(3)$} (1,3)  --  node[above right]{$L(4)$}cycle;
\draw[dashed,orange] (-4,1.5) to[out=350,in=120] (-3,1);
\draw[dashed,orange] (1,1.5) to[out=350,in=120] (2,1);
\fill[black](-4,0)circle(0.06);
\fill[black](-1,0)circle(0.06);
\fill[black](-4,3)circle(0.06);
\fill[black](1,0)circle(0.06);
\fill[black](4,0)circle(0.06);
\fill[black](1,3)circle(0.06);
\draw(-4,0) node[left]{$\begin{pmatrix}-2\\0\end{pmatrix}$};
\draw(-1,0) node[right]{$\begin{pmatrix}1\\0\end{pmatrix}$};
\draw(-4,3) node[left]{$\begin{pmatrix}0\\1\end{pmatrix}$};
\draw(1,0) node[left]{$\begin{pmatrix}0\\-2\end{pmatrix}$};
\draw(4,0) node[right]{$\begin{pmatrix}2\\-2\end{pmatrix}$};
\draw(1,3) node[left]{$\begin{pmatrix}-2\\2\end{pmatrix}$};
\end{tikzpicture}
	    \caption{The fiber coordinate of $\bL_{1,2,2}$. The left side is the first sheet $\bL_{1,2,2}^{(1)}$ and the right side is the second sheet $\bL_{1,2,2}^{(2)}$ of $\bL_{1,2,2}$. In a neighborhood of each edge of the moment polytope, $\bL_{1,2,2}$ behaves like $L(k)$.}
		\label{fig:fib coord of L122}
\end{figure}
Then, the intersections $L_{\cO(D_{1}-2D_{2})} \cap \bL_{1,2,2}$, $\bL_{1,2,2} \cap L_{\cO(D_{0}+3D_{1})}$, and $L_{\cO(D_{1}-2D_{2})} \cap L_{\cO(D_{0}+3D_{1})}$ are all transverse. 
The projections $U=\pi(L_{\cO(D_{1}-2D_{2}} \cap \bL_{1,2,2}^{(2)})$, $V=\pi(\bL_{1,2,2}^{(1)} \cap L_{\cO(D_{0}+3D_{1})})$, and $W=\pi(L_{\cO(D_{1}-2D_{2})} \cap L_{\cO(D_{0}+3D_{1})})$ are given by
\begin{align*}
U &= \left(\frac{1}{8}, \dfrac{7}{8}\right), & V &= \left(1,0\right),  & W &= \left(\frac{2}{5}, \dfrac{2}{5}\right)
\end{align*}
In particular, these points turn out to be the morphisms of degree $0$ in $\Mo^{\mult}(P)$. For these morphisms, there is no normal gradient tree, but there is one non-trivial jagged gradient tree. Figure \ref{fig:non-trivial jagged gradient tree in P} shows the non-trivial jagged gradient tree $\gamma$. 
\begin{figure}[h]
		\centering
\begin{tikzpicture}[scale=2]
\draw(4,0)  --  (1,0) --  (1,3)  -- cycle;
\draw[dashed,orange] (1,1.5) to[out=340,in=120] (2,1);
\fill[black](2.15,1.15)circle(0.04);;
\draw(2.3,1.3) node{$W$};
\fill[black](4,0)circle(0.04) node[right]{$V$};
\fill[black](1.5,2.5)circle(0.04) node[above right]{$U$};
\fill[black](2,1)circle(0.03);
\draw[-stealth,postaction=decorate,
    decoration={markings,
    mark = between positions 0.2 and 1 step 0.3
    with {\arrow{stealth}}}](2,1)to[out=240,in=140](2.5,0) node[below]{$(21)$-wall};
\draw[-stealth,postaction=decorate,
    decoration={markings,
    mark = between positions 0.2 and 1 step 0.3
    with {\arrow{stealth}}}](2,1)to[out=100,in=340](1,1.75) node[left]{$(21)$-wall};
\draw[-stealth,postaction=decorate,
    decoration={markings,
    mark = between positions 0.2 and 1 step 0.3
    with {\arrow{stealth}}}](2,1)to[out=340,in=280](2.5,1.5) node[above right]{$(12)$-wall};
\draw[blue, thick,-stealth,postaction=decorate,
    decoration={markings,    
    mark = between positions 0.2 and 1 step 0.4
    with {\arrow{stealth}}}](1.5,2.5)to[out=280,in=130](2.15,1.15);
\fill[red](2.3,1.05)circle(0.03);
\draw[red, thick,-stealth,postaction=decorate,
    decoration={markings,
    mark = between positions 0.2 and 1 step 0.4
    with {\arrow{stealth}}}](4,0)to[out=160,in=310](2.3,1.05);
\draw[red, thick,-stealth,postaction=decorate,
    decoration={markings,
    mark = between positions 1 and 1 step 1
    with {\arrow{stealth}}}](2.3,1.05)to[out=160,in=320](2.15,1.15);
\end{tikzpicture}
	    \caption{The non-trivial jagged gradient tree $\gamma$. The blue line is the ordinary gradient line of $\grad(f^{(2)}_{\bL_{1,2,2}} - f_{L_{\cO(D_1-2D_2)}})$. The red line is the jagged gradient line. It first moves along $\grad(f_{L_{\cO(D_0 + 3D_1)}} - f^{(1)}_{\bL_{1,2,2}})$, and when it hits $(12)$-wall, it moves along $\grad(f_{L_{\cO(D_0 + 3D_1)}} - f^{(2)}_{\bL_{1,2,2}})$.}
		\label{fig:non-trivial jagged gradient tree in P}
\end{figure}
Thus, the product $\fm_2(U,V) = \pm e^{-A(\gamma)}W \neq 0$ is also non-trivial in $\Mo^{\mult}(P)$. This is the mirror description of the composition $\beta \circ \alpha \neq 0$ in $\DG^\vect(\cp)$. Under these observations, we expect that $\Mo^\mult(P)$ and $\DG^\vect(\cp)$ are $A_\infty$-equivalent.

\appendix
\section{Local model of Lagrangian multi-sections}\label{local model}
In this appendix, we review the local model of a Lagrangian multi-section around the simple branch point. For more details see \cite{Fuk05,Sue21,Sue24}.\par
Let $B:=N_\R=\R^2\cong\C$ and $Y:=T^*\!N_\R\cong\C^2$. We denote by $(\xi_1,\xi_2)$ the real coordinates of $B$ and $(\xi_1,\xi_2,y^1,y^2)$ the real coordinates of $Y$. The standard symplectic form of $Y$ is given by $\omega=d\xi_1\wedge dy^1+d\xi_2\wedge dy^2$. Let $\xi:=\xi_1+\rt \xi_2$, $\eta:=y^1+\rt y^2$ be the complex coordinates of $Y$. 
The Lagrangian submanifold $L_2$ is defined by
\begin{equation*}
L_2:=\{(\xi,\eta)\in\C^2\ |\ \xi\in B,\ \bar{\xi}=\eta^2 \}.
\end{equation*}
Indeed, we have 
\begin{equation*}
\omega|_{L_2}=(2y^1dy^1-2y^2dy^2)\wedge dy^1 + (-2y^2dy^1-2y^1dy^2)\wedge dy^2=0
\end{equation*}
since $\xi_1=(y^1)^2-(y^2)^2$ and $ \xi_2=-2y^1y^2$. If we use the polar coordinate $\xi=re^{\rt\theta}$, then the Lagrangian $L_2$ is expressed as
\begin{equation*}
L_2=\left\{\left(re^{\rt\theta},\sqrt{r}e^{-\rt\frac{\theta}{2}}\right)\in\C^2\ \middle|\ r\in\R_{\geq0},\ \theta\in\R\right\}.
\end{equation*}\par
On the other hand, we consider the following function $f_{L_2}$ on $B$ and the graph of $df_{L_2}$
\begin{equation*}
\mathrm{graph}(df_{L_2})=\left\{ (\xi,df_{L_2}(\xi))\in\C^2\ \middle|\ \xi=re^{\rt\theta}\in\C,\ f_{L_2}(r,\theta)=\frac{2}{3}r^{\frac{3}{2}}\cos\frac{3}{2}\theta \right\}.
\end{equation*}
By the polar coordinate $\xi=re^{\rt\theta}=\xi_1+\rt \xi_2$, we have
\begin{equation*}
d\xi_1=\cos\theta dr -r\sin\theta d\theta,\ \ \ d\xi_2=\sin\theta dr +r\cos\theta d\theta,
\end{equation*}
and
\begin{equation*}
df_{L_2} = r^{\frac{1}{2}}\cos\frac{3\theta}{2} dr -r^{\frac{3}{2}}\sin\frac{3\theta}{2} d\theta
=r^{\frac{1}{2}}\cos\frac{\theta}{2}d\xi_1-r^{\frac{1}{2}}\sin\frac{\theta}{2}d\xi_2.
\end{equation*}
Thus, the fiber coordinates are given by $(y^1,y^2) = (r^{\frac{1}{2}}\cos\frac{\theta}{2},-r^{\frac{1}{2}}\sin\frac{\theta}{2})$ and we have
\begin{equation*}
\eta^2=(y^1+\rt y^2)^2= \left(r^{\frac{1}{2}}\cos\frac{\theta}{2}-\rt r^{\frac{1}{2}}\sin\frac{\theta}{2}\right)^2
=\left(r^{\frac{1}{2}}e^{-\rt\frac{\theta}{2}}\right)^2=re^{-\rt\theta}=\bar{\xi}.
\end{equation*}
Namely, we obtain
\begin{equation*}
\mathrm{graph}(df_{L_2})=L_2.
\end{equation*}\par
Next, we consider the gradient vector field of the function $f_{L_2}$.
We recall that the flat metric with respect to $(\xi_1,\xi_2)$ is given by $\begin{pmatrix}1&0\\0&1\end{pmatrix}$ and with respect to $(r,\theta)$ is given by $\begin{pmatrix}1&0\\0&r^2\end{pmatrix}$. Hence, we get
\begin{equation*}
\grad(f_{L_2}) = r^{\frac{1}{2}}\cos\frac{3\theta}{2}\frac{\del}{\del r} - r^{-\frac{1}{2}}\sin\frac{3\theta}{2}\frac{\del}{\del \theta}
\ \ \ \Big(= r^{\frac{1}{2}}\cos\frac{\theta}{2}\frac{\del}{\del \xi_1} - r^{\frac{1}{2}}\sin\frac{\theta}{2}\frac{\del}{\del \xi_2}\Big).
\end{equation*}
By this expression, we obtain three gradient lines which are incoming to the origin or outgoing from the origin. The direction of these gradient lines depends on the branch of $f_{L_2}$. Figure \ref{fig line} shows the gradient vector fields corresponding to each branch.\par
\begin{figure}[h]
			\centering
			\includegraphics[width=75mm]{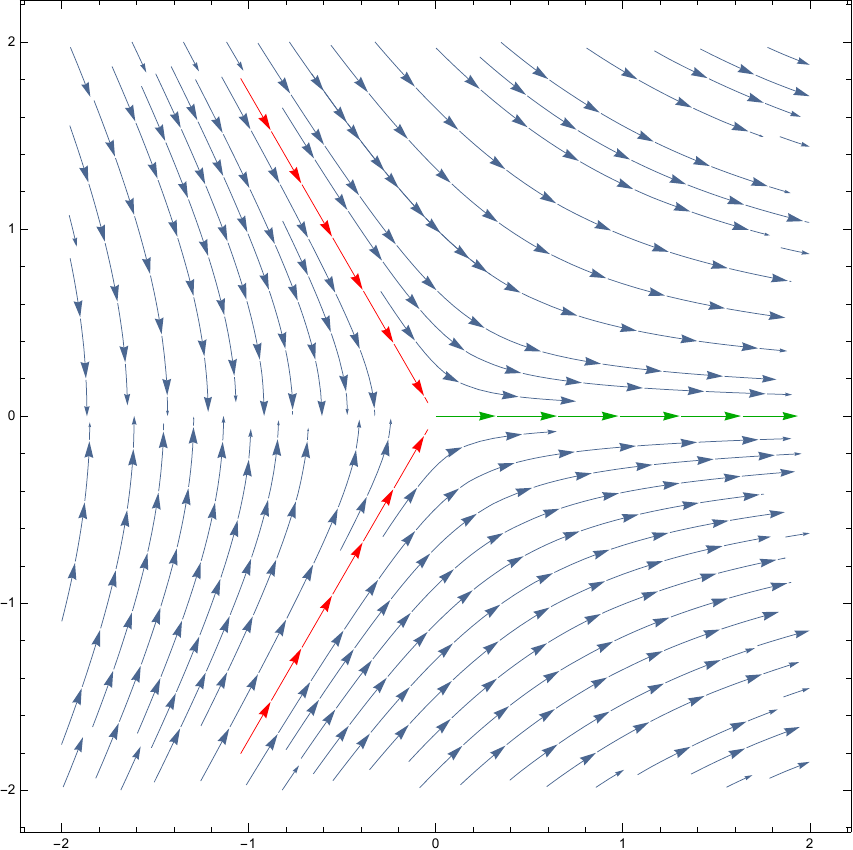}
			\includegraphics[width=75mm]{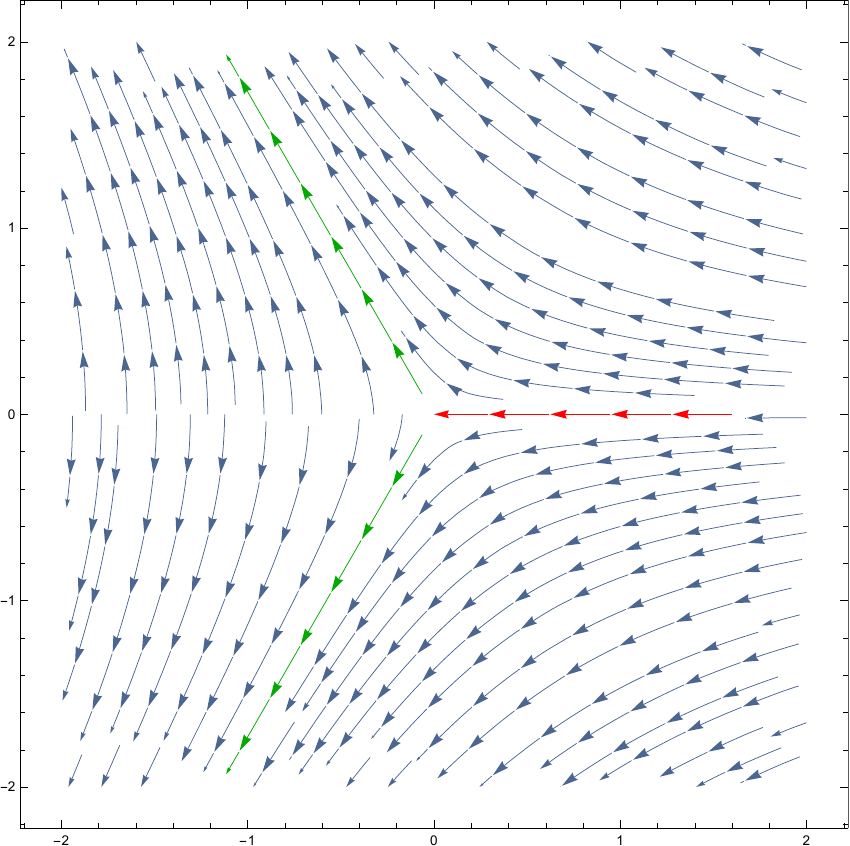}
			\caption{The gradient vector field of $f_{L_2}$ in $B$. The one on the left is for $-\pi<\theta<\pi$, and the one on the right is for $\pi<\theta<3\pi$.}
			\label{fig line}
\end{figure}
Finally, we consider the case of the Fubini-Study metric and the behavior of the gradient vector field on $\Int(P)$. Namely, the metric on $N_\R$ is given by
\begin{equation*}
\begin{pmatrix}
g^{11} & g^{12}\\
g^{21} & g^{22}
\end{pmatrix}
=
\begin{pmatrix}
\dfrac{2e^{2\xi_1}(1+e^{2\xi_2})}{(1+e^{2\xi_1}+e^{2\xi_2})^2} & \dfrac{- 2e^{2\xi_1}e^{2\xi_2}}{(1+e^{2\xi_1}+e^{2\xi_2})^2} \\
\dfrac{-2e^{2\xi_1}e^{2\xi_2}}{(1+e^{2\xi_1}+e^{2\xi_2})^2} & \dfrac{2e^{2\xi_2}(1+e^{2\xi_1})}{(1+e^{2\xi_1}+e^{2\xi_2})^2}
\end{pmatrix}.
\end{equation*}
We use the Legendre transformation $(\xi_1,\xi_2)\leftrightarrow(x^1,x^2)$ given in subsection \ref{SYZ construction}. Under the Legendre transformation, the metric on $\Int(P)$ is given by the inverse matrix of $\{g^{ij}\}$.
If we consider the function $f_{L_2}$ as the function of variables $(x^1,x^2)=\left(\frac{e^{2 \xi_1}}{1+e^{2 \xi_1}+e^{2 \xi_2}}, \frac{e^{2 \xi_2}}{1+e^{2 \xi_1}+e^{2 \xi_2}}\right)$, then the gradient vector field on $\Int(P)$ is given by
\begin{equation*}
\grad(f_{L_2}) = \sum_{i,j=1}^2\frac{\del f_{L_2}}{\del x^i}g^{ij}\frac{\del}{\del x^j} = \frac{\del f_{L_2}}{\del \xi_1}\frac{\del}{\del x^1} + \frac{\del f_{L_2}}{\del \xi_2}\frac{\del}{\del x^2} = r^{\frac{1}{2}}\cos\frac{\theta}{2}\frac{\del}{\del x^1} - r^{\frac{1}{2}}\sin\frac{\theta}{2}\frac{\del}{\del x^2}.
\end{equation*}
Note that $r$ and $\theta$ are the functions with$(x^1,x^2)$ as variables under the Legendre transformation. Figure \ref{fig line poly} shows the gradient vector field of $f_{L_2}$ in the moment polytope $P$.
\begin{figure}[h]
			\centering
			\includegraphics[width=75mm]{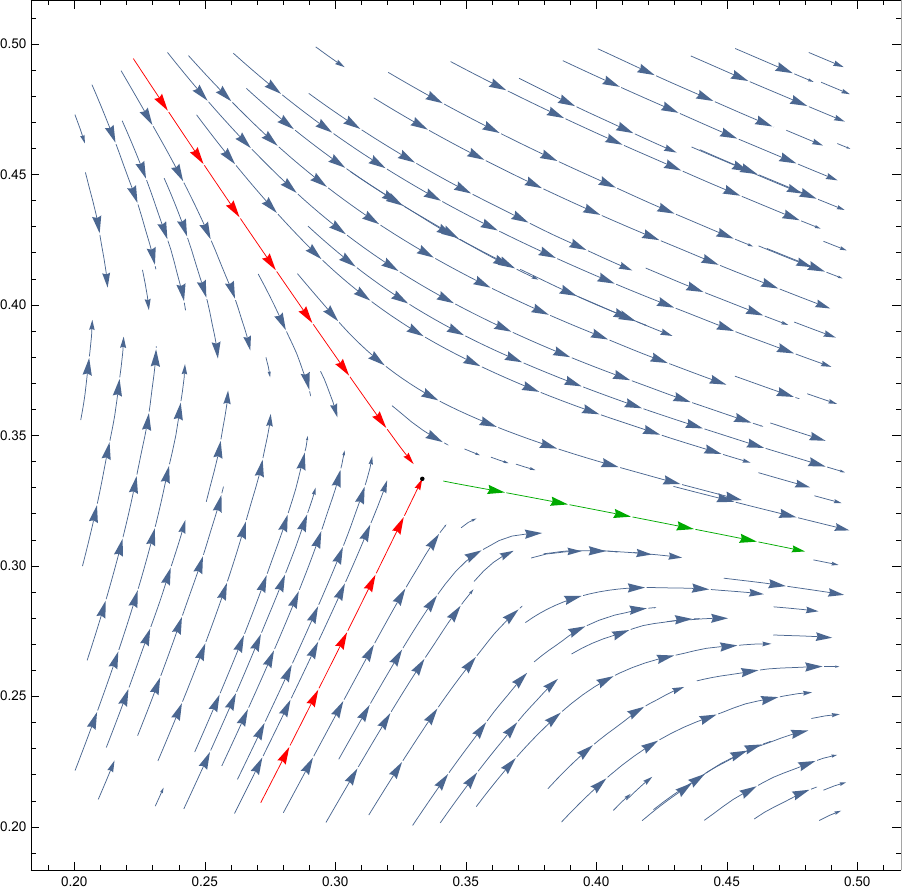}
			\includegraphics[width=75mm]{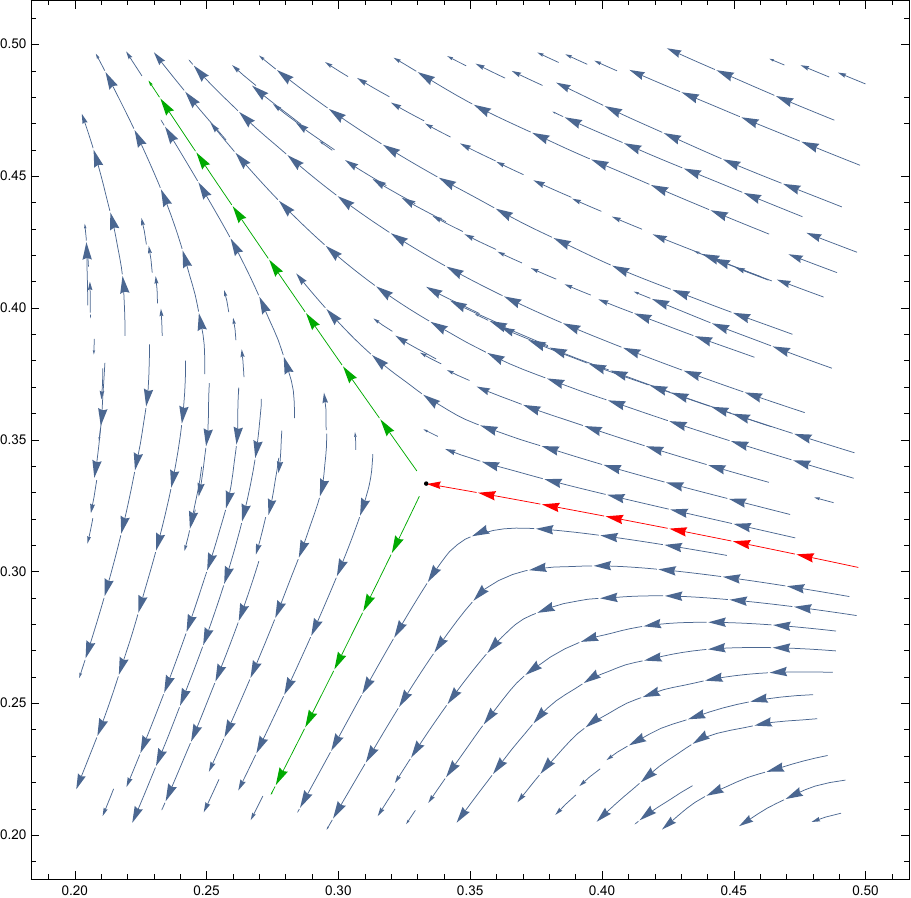}
			\caption{The gradient lines of $f_{L_2}$ in the moment polytope $P$. The one on the left is for $-\pi<\theta<\pi$, and the one on the right is for $\pi<\theta<3\pi$.}
			\label{fig line poly}
\end{figure}

\section{Non-trivial jagged gradient tree in the local model}\label{Non-trivial jagged gradient tree in the local model}
In this appendix, we give an example of a non-trivial jagged gradient tree in the local model of Lagrangian multi-sections. We use the same notation as the previous appendix. Namely, let $B:=N_\R=\R^2\cong\C$ and $Y:=T^*\!N_\R\cong\C^2$. We denote by $(\xi_1,\xi_2)$ the real coordinates of $B$ and $(\xi_1,\xi_2,y^1,y^2)$ the real coordinates of $Y$. The standard symplectic form of $Y$ is given by $\omega=d\xi_1\wedge dy^1+d\xi_2\wedge dy^2$. Let $\xi_1=r\cos\theta$, $\xi_2=r\sin\theta$ be the polar coordinates of $B$. \par
We set the functions $f_1, f_2, f_3$ on $B$ as follows:
\begin{equation*}
\begin{split}
f_1 &:= - 3(1+\sqrt{3})(\xi_1^2 + \xi_2^2) + 6\xi_1 + 6\sqrt{3}\xi_2,\\
f_2 &:= 4\sqrt{2}r^{\frac{3}{2}}\cos\frac{3}{2}\theta,\\
f_3 &:=  3(1+\sqrt{3})(\xi_1^2 + \xi_2^2) - k\xi_1 + \ell\xi_2,
\end{split}
\end{equation*}
where $k$ and $\ell$ are positive real numbers satisfying the following equations:
\begin{equation*}
\left\{ 
\begin{alignedat}{2}  
6\ell &= (1+\sqrt{3})k^2 - 6k, \\
24(k+\sqrt{3}\ell) &= (1+\sqrt{3})(\ell-\sqrt{3}k)^2.
\end{alignedat}
\right.
\end{equation*}
The graph of $1$-forms $\bL_1=\mathrm{graph}(df_1)$ and $\bL_3=\mathrm{graph}(df_3)$ are Lagrangian sections and $\bL_2=\mathrm{graph}(df_2)$ is the Lagrangian multi-section. We denote by $\bL_2=\bL_2^{(1)}\cup\bL_2^{(2)}=\mathrm{graph}(df_2^{(1)})\cup\mathrm{graph}(df_2^{(2)})$ the sheets\footnote{For $-\pi<\theta<\pi$, we have $f_2^{(1)}(r,\theta)=4\sqrt{2}r^{\frac{3}{2}}\cos\frac{3}{2}\theta$ and $f_2^{(2)}(r,\theta)=4\sqrt{2}r^{\frac{3}{2}}\cos\frac{3}{2}(\theta+2\pi)$.} of $\bL_2$. In particular, $\bL_1$ and $\bL_3$ are affine Lagrangians and $\bL_2$ is the local model of Lagrangian multi-section. Then, the critical points are given by
\begin{align*}
\mathrm{Crit}(f_2-f_1) &:= \left\{u_1:=(0,1),\ u_2:=\left(\frac{\sqrt{3}}{2},\frac{1}{2}\right)\right\},\\
\mathrm{Crit}(f_3-f_2) &:= \left\{v_1:=\left(\frac{\sqrt{3}(k+\sqrt{3}\ell)}{24(1+\sqrt{3})},-\frac{k+\sqrt{3}\ell}{24(1+\sqrt{3})}\right),\ v_2:=\left(0,-\frac{k^2}{36}\right)\right\},\\
\mathrm{Crit}(f_3-f_1) &:= \left\{w:=\left(\frac{6+k}{12(1+\sqrt{3})},-\frac{6\sqrt{3}-\ell}{12(1+\sqrt{3})}\right)\right\}.
\end{align*}
In the polar coordinates $(r,\theta)$, we have 
\begin{align*}
u_1 &= \left(1,\frac{\pi}{2}\right), & u_2 &= \left(1,\frac{\pi}{6}+2\pi\right),\\
v_1 &= \left(\frac{k+\sqrt{3}\ell}{12(1+\sqrt{3})},\frac{-\pi}{3}\right), & v_2 &= \left(\frac{k^2}{36},\frac{3\pi}{2}\right).
\end{align*}
Namely, each $u_i$ is the projection of the intersection $p_i:=\bL_1\cap\bL_2^{(i)}$, each $v_i$ is the projection of $q_i:=\bL_2^{(i)}\cap\bL_3$, and $w$ is the projection of $r_w:=\bL_1\cap\bL_3$ to $B$. Furthermore, these critical points are non-degenerate and the Morse indices are zero. We consider the gradient trees starting at $u_i,v_j$ and ending at $w$. Firstly, there are two ordinary gradient trees. These correspond to a pseudo-holomorphic disc bounded by $\bL_1$, $\bL_2^{(1)}$, and $\bL_3$, and to a pseudo-holomorphic disc bounded by $\bL_1$, $\bL_2^{(2)}$, and $\bL_3$. Figure \ref{fig: ordinary gradient trees} shows these ordinary gradient trees and Figure \ref{fig: ordinary triangle} shows the pseudo-holomorphic disc corresponding to the ordinary gradient tree starting at $u_1, v_1$ and ending at $w$.
\begin{figure}[h]
			\centering
			\includegraphics[width=75mm]{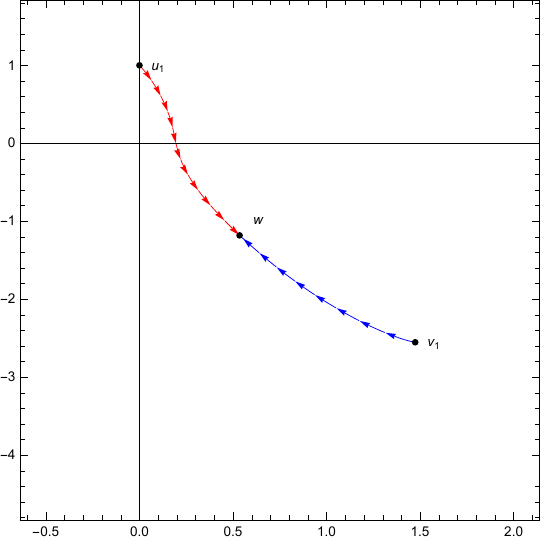}
			\includegraphics[width=75mm]{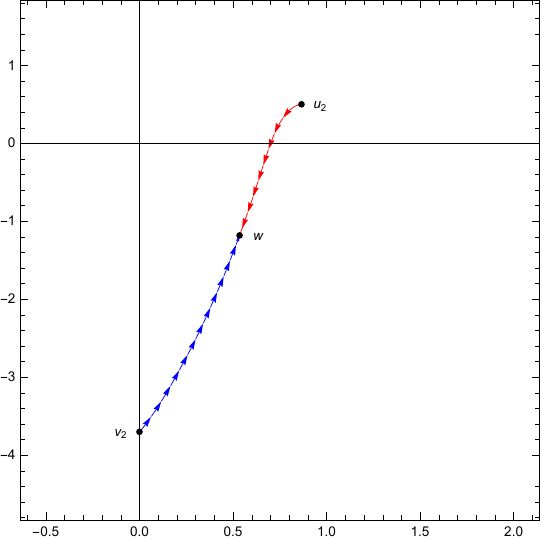}
			\caption{The ordinary gradient trees.}
			\label{fig: ordinary gradient trees}
\end{figure}
\begin{figure}[h]
			\centering
\begin{tikzpicture}[scale=1.3]
\draw[->,>=stealth](-1,0)--  (3,0)node[right]{$\xi_1$-axis};
\draw[->,>=stealth](0,-2)--  (0,2)node[above]{$y^1$-axis};
\fill[red,opacity=0.2] (19/22,-2/11) -- plot[domain={(121 - sqrt(5425))/128}:{19/22}, samples=100] (\x,{sqrt(\x)}) --   cycle;
\fill[cyan,opacity=0.2] (19/22,-2/11) -- (19/22,{sqrt(19/22)}) -- plot[domain={19/22}:{(41 + sqrt(657))/32}, samples=100] (\x,{sqrt(\x)}) --   cycle;
\draw[samples=1000,domain=0:3,red] plot(\x,{sqrt(\x)});
\draw[samples=1000,domain=0:3,blue] plot(\x,{-sqrt(\x)});
\draw[samples=1000,domain=-1:0,red] plot(\x,{sqrt(-\x)});
\draw[samples=1000,domain=-1:0,blue] plot(\x,{-sqrt(-\x)});
\draw[green] (-0.5,2) -- (2,-2) node[below]{$\mathbb{L}_1$};
\draw[violet] (-0.5,-2) -- (2.5,2)node[above]{$\mathbb{L}_3$};
\draw[red] (3,{sqrt(3)}) node[right]{$\mathbb{L}_2^{(1)}$};
\draw[blue] (3,{-sqrt(3)}) node[right]{$\mathbb{L}_2^{(2)}$};
\fill[black](19/22,-2/11)circle(0.04) node[right]{$r_w$};
\fill[black]({(121 - sqrt(5425))/128},{sqrt((121 - sqrt(5425))/128)})circle(0.04);
\draw ({(121 - sqrt(5425))/128 + 0.05},{sqrt((121 - sqrt(5425))/128 +0.2)}) node[above]{$p_1$};
\fill[black]({(41 + sqrt(657))/32},{sqrt((41 + sqrt(657))/32)})circle(0.04) node[below right]{$q_1$};
\fill[black]({(121 + sqrt(5425))/128},{-sqrt((121 + sqrt(5425))/128)})circle(0.04) node[above right]{$p_2$};
\fill[black]({(41 - sqrt(657))/32},{-sqrt((41 - sqrt(657))/32)})circle(0.04);
\draw ({(41 - sqrt(657))/32},{-sqrt((41 - sqrt(657))/32)-0.07}) node[below]{$q_2$};
\end{tikzpicture}\ \ \ \ \ \ 
\begin{tikzpicture}[scale=1.3]
\draw[->,>=stealth](-2,0)--  (2,0)node[right]{$\xi_2$-axis};
\draw[->,>=stealth](0,-2)--  (0,2)node[above]{$y^2$-axis};
\fill[red,opacity=0.2] (-9/74,92/74) -- ({1},{-1})  -- ({-9/74},{9/74})  --cycle;
\fill[cyan,opacity=0.2] (-9/74,92/74) -- (-9/74,9/74)  -- ({-26/57},{26/57})   --cycle;
\draw[red] (-2,2) -- (2,-2) node[right]{$\mathbb{L}_2^{(1)}$};
\draw[blue] (-2,-2) -- (2,2) node[right]{$\mathbb{L}_2^{(2)}$};
\draw[green] (-0.5,2) -- (1.5,-2) node[below]{$\mathbb{L}_1$};
\draw[violet] (-1.5,-2) node[right]{$\mathbb{L}_3$} -- (0.2,2);
\fill[black](-9/74,92/74)circle(0.04) node[left]{$r_w$};
\fill[black]({1},{-1})circle(0.04) node[left]{$p_1$};
\fill[black]({-26/57},{26/57})circle(0.04) node[left]{$q_1$};
\fill[black]({1/3},{1/3})circle(0.04)node[right]{$p_2$};
\fill[black]({-26/23},{-26/23}) circle(0.04) node[right]{$q_2$};
\end{tikzpicture}
\caption{The pseudo-holomorphic disc corresponding to the ordinary gradient tree.}
			\label{fig: ordinary triangle}
\end{figure}

Next, we consider the jagged gradient trees. Note that the positive part of the $\xi_1$-axis is the $(12)$-wall of the spectral network subordinate to $\bL_2$, as shown in Figure \ref{fig line}.
Figure \ref{fig: example of jagged ordinary gradient trees} shows the jagged gradient tree $\gamma$ and Figure \ref{fig: jagged triangle} shows the pseudo-holomorphic disc $\varphi$ corresponding to the jagged gradient tree starting at $u_2, v_1$ and ending at $w$.
\begin{figure}[h]
			\centering
			\includegraphics[width=75mm]{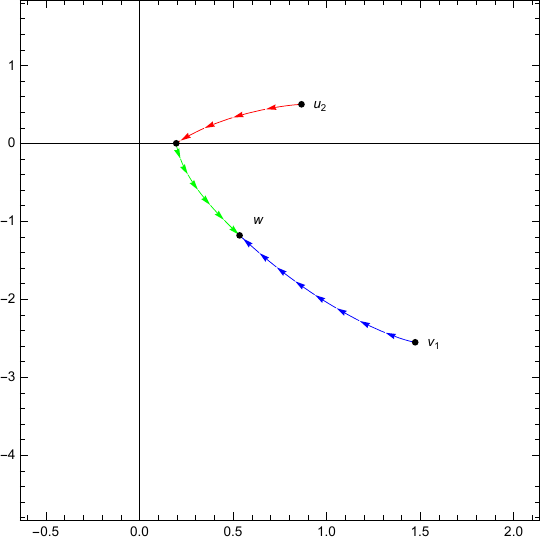}
			\caption{The jagged gradient trees. Starting from $u_2$, it proceeds along $\grad(f_2^{(2)}-f_1)$, interacts with the $(12)$-wall, proceeds along $\grad(f_2^{(1)}-f_1)$, and ends at $w$.}
			\label{fig: example of jagged ordinary gradient trees}
\end{figure}
\begin{figure}[h]
			\centering
\begin{tikzpicture}[scale=1.3]
\draw[->,>=stealth](-1,0)--  (3,0)node[right]{$\xi_1$-axis};
\draw[->,>=stealth](0,-2)--  (0,2)node[above]{$y^1$-axis};
\fill[orange,opacity=0.2] (0,0) -- plot[domain=0:{(121 - sqrt(5425))/128}, samples=100] (\x,{sqrt(\x)}) --  ({(121 - sqrt(5425))/128},{-sqrt((121 - sqrt(5425))/128)}) -- plot[domain={(121 - sqrt(5425))/128}:0, samples=100] (\x,{-sqrt(\x)}) -- cycle;
\fill[red,opacity=0.2] ({(121 + sqrt(5425))/128},{-sqrt((121 + sqrt(5425))/128)}) -- plot[domain={(121 + sqrt(5425))/128}:{(121 - sqrt(5425))/128}, samples=100] (\x,{-sqrt(\x)}) -- ({(121 - sqrt(5425))/128},{sqrt((121 - sqrt(5425))/128)}) --   cycle;
\fill[cyan,opacity=0.2] (19/22,-2/11) -- (19/22,{sqrt(19/22)}) -- plot[domain={19/22}:{(41 + sqrt(657))/32}, samples=100] (\x,{sqrt(\x)}) --   cycle;
\fill[green,opacity=0.2] ({(121 - sqrt(5425))/128},{sqrt((121 - sqrt(5425))/128)}) -- plot[domain={(121 - sqrt(5425))/128}:19/22, samples=100] (\x,{sqrt(\x)}) -- (19/22,-2/11) --cycle;
\draw[samples=1000,domain=0:3,red] plot(\x,{sqrt(\x)});
\draw[samples=1000,domain=0:3,blue] plot(\x,{-sqrt(\x)});
\draw[samples=1000,domain=-1:0,red] plot(\x,{sqrt(-\x)});
\draw[samples=1000,domain=-1:0,blue] plot(\x,{-sqrt(-\x)});
\draw[green] (-0.5,2) -- (2,-2) node[below]{$\mathbb{L}_1$};
\draw[violet] (-0.5,-2) -- (2.5,2)node[above]{$\mathbb{L}_3$};
\draw[red] (3,{sqrt(3)}) node[right]{$\mathbb{L}_2^{(1)}$};
\draw[blue] (3,{-sqrt(3)}) node[right]{$\mathbb{L}_2^{(2)}$};
\fill[black](19/22,-2/11)circle(0.04) node[right]{$r_w$};
\fill[black]({(41 + sqrt(657))/32},{sqrt((41 + sqrt(657))/32)})circle(0.04) node[below right]{$q_1$};
\fill[black]({(121 + sqrt(5425))/128},{-sqrt((121 + sqrt(5425))/128)})circle(0.04) node[above right]{$p_2$};
\end{tikzpicture}\ \ \ \ \ \ 
\begin{tikzpicture}[scale=1.3]
\draw[->,>=stealth](-2,0)--  (2,0)node[right]{$\xi_2$-axis};
\draw[->,>=stealth](0,-2)--  (0,2)node[above]{$y^2$-axis};
\fill[red,opacity=0.2] (0,1) -- ({1/3},{1/3})  -- (0,0)  --cycle;
\fill[green,opacity=0.2] (-9/74,92/74) -- (0,1)  -- (0,0) -- (-9/74,9/74)  --cycle;
\fill[cyan,opacity=0.2] (-9/74,92/74) -- (-9/74,9/74) -- ({-26/57},{26/57})  --cycle;
\draw[red] (-2,2) -- (2,-2) node[right]{$\mathbb{L}_2^{(1)}$};
\draw[blue] (-2,-2) -- (2,2) node[right]{$\mathbb{L}_2^{(2)}$};
\draw[green] (-0.5,2) -- (1.5,-2) node[below]{$\mathbb{L}_1$};
\draw[violet] (-1.5,-2) node[right]{$\mathbb{L}_3$} -- (0.2,2);
\fill[black](-9/74,92/74)circle(0.04) node[left]{$r_w$};
\fill[black]({-26/57},{26/57})circle(0.04) node[left]{$q_1$};
\fill[black]({1/3},{1/3})circle(0.04)node[right]{$p_2$};
\end{tikzpicture}
\caption{The pseudo-holomorphic disc corresponding to the jagged gradient tree. The orange area is the contribution from the $(12)$-wall of the spectral network subordinate to $\bL_2$.}
			\label{fig: jagged triangle}
\end{figure}
We denote by $x$ the interaction point between the gradient line and $(12)$-wall. Then, the symplectic area $A(\gamma)$ of the pseudo-holomorphic disc $\varphi$ corresponding to the jagged gradient tree $\gamma$ is given by
\begin{align}
A(\gamma) &= \int_{D^2} \varphi^*\omega \notag\\
&= \int_{u_2}^{x} d(f_2^{(2)}-f_1) + \int_{0}^{x}d(f_2^{(1)}-f_2^{(2)}) + \int_{x}^{w} d(f_2^{(1)}-f_1) + \int_{v_1}^{w} d(f_3-f_2^{(1)}) \label{int of area}\\
&= \left(f_3(w)-f_1(w)\right) - \left(f_2^{(2)}(u_2)-f_1(u_2)\right) - \left(f_3(v_1)-f_2^{(1)}(v_1)\right). 
\end{align}
Note that the second term in Eq. \eqref{int of area} is the contribution from the $(12)$-wall (see Figure \ref{fig: jagged triangle}). 
In particular, the area $A(\gamma) \approx 121.1$ is positive since $k \approx 11.5$ and $\ell \approx 49.1$.\par
In the proof of the main theorem of this paper, jagged gradient trees did not appear. However, in general, it is necessary to consider pseudo-holomorphic discs containing the ramification locus as in this example and section \ref{Example of nontrivial jagged gradient tree in CP2}. Therefore, jagged gradient trees are required in the definition of the category $\Mo^{\mult}(P)$ of multi-valued Morse homotopy.


\end{document}